\documentclass{article}

\usepackage[square,sort&compress,numbers]{natbib}
\usepackage[english]{babel}
\usepackage{fullpage}

\usepackage{hyperref}

\usepackage{enumitem}
\setenumerate[0]{label=\arabic*)}
\setlist{nosep}
\setlist{leftmargin=*}

\usepackage{amssymb,amsmath,amsfonts}
\usepackage{amsthm}
\usepackage[all]{xy}
\usepackage[capitalise]{cleveref}

\allowdisplaybreaks

\usepackage{xcolor}
\definecolor{MyBlue}{cmyk}{1,0.13,0,0.63}
\definecolor{MyGreen}{cmyk}{0.91,0,0.88,0.52}
\newcommand{\mylinkcolor}{MyBlue}
\newcommand{\mycitecolor}{MyGreen}
\newcommand{\myurlcolor}{webbrown}

\makeatletter
\def\@endtheorem{\endtrivlist}
\makeatother
\theoremstyle{plain}
\newtheorem{thm}{Theorem}[section]
\newtheorem{lem}[thm]{Lemma}
\newtheorem{prop}[thm]{Proposition}
\newtheorem{coro}[thm]{Corollary}

\theoremstyle{definition}
\newtheorem{defn}[thm]{Definition}
\newtheorem{remark}[thm]{Remark}
\newtheorem{assumption}[thm]{Assumption}

\newtheorem{example}[thm]{Example}

\newtheoremstyle{note}
{3pt}
{3pt}
{\bfseries}
{\parindent}
{\bfseries\itshape}
{:}
{.5em}
{}
\theoremstyle{note}
\newtheorem*{Question}{Question}
\newtheorem*{Problem}{Problem}
\newtheorem*{Note}{Note}

\renewcommand{\eqref}[1]{\labelcref{#1}}
\crefname{thm}{Theorem}{Theorems}
\crefname{lem}{Lemma}{Lemmas}
\crefname{prop}{Proposition}{Propositions}
\crefname{coro}{Corollary}{Corollaries}
\crefname{defn}{Definition}{Definitions}
\crefname{example}{Example}{Examples}
\crefname{remark}{Remark}{Remarks}

\hypersetup{%
  bookmarksnumbered=true,bookmarksopen=false,%
  plainpages=false,
  linktocpage=true,%
  colorlinks=true,breaklinks=true,%
  linkcolor=\mylinkcolor,citecolor=\mycitecolor,urlcolor=\myurlcolor,%
  pdfpagelayout=OneColumn,%
  pageanchor=true,%
}

\RequirePackage{slashed}
\newcommand{\sD}{\slashed{D}}

\newcommand{\N}{\mathbb{N}}
\newcommand{\Z}{\mathbb{Z}}
\newcommand{\R}{\mathbb{R}}
\newcommand{\C}{\mathbb{C}}

\newcommand{\mH}{\mathcal{H}}

\newcommand{\D}{\mathcal{D}}
\newcommand{\A}{\mathcal{A}}
\newcommand{\E}{\mathcal{E}}
\newcommand{\mW}{\mathcal{W}}
\newcommand{\mL}{\mathcal{L}}
\newcommand{\mJ}{\mathcal{J}}
\newcommand{\mS}{\mathcal{S}}
\newcommand{\mK}{\mathcal{K}}

\DeclareMathOperator{\End}{End}
\DeclareMathOperator{\Dom}{Dom}

\DeclareMathOperator{\dist}{dist}

\newcommand{\bundlefont}[1]{{\mathtt{#1}}}
\newcommand{\bS}{\bundlefont{S}}
\newcommand{\bE}{\bundlefont{E}}

\newcommand{\til}[1]{\widetilde{#1}}
\newcommand{\la}{\langle}
\newcommand{\ra}{\rangle}
\newcommand{\op}{\textnormal{op}}
\newcommand{\into}{\hookrightarrow}
\newcommand{\Cliff}{{\mathrm{Cl}}}
\newcommand{\CCliff}{{\mathbb{C}\mathrm{l}}}
\newcommand{\dvol}{\textnormal{dvol}}

\renewcommand{\Re}{\mathop{\textnormal{Re}}}
\renewcommand{\Im}{\mathop{\textnormal{Im}}}
\renewcommand{\bar}[1]{\overline{#1}}

\newcommand{\mattwo}[4]{
  \left(\!\!\!\begin{array}{c@{~}c}#1&#2\\#3&#4\\\end{array}\!\!\!\right)
}

\title{Indefinite Kasparov modules and pseudo-Riemannian manifolds}
\author{
Koen van den Dungen$^{1,2}$\footnote{Email: \texttt{koen.vandendungen@anu.edu.au}}$\phantom{x}$
and Adam Rennie$^2$\footnote{Email: \texttt{renniea@uow.edu.au}}\\[4mm]
{\normalsize ${}^1$Mathematical Sciences Institute}, 
{\normalsize Australian National University}\\
{\normalsize Canberra, ACT 0200, Australia}\\[2mm]
{\normalsize ${}^2$School of Mathematics and Applied Statistics},
{\normalsize University of Wollongong}\\
{\normalsize Wollongong, NSW 2522, Australia}\\[2mm]
}

\begin{document}
\date{}
\maketitle

\begin{abstract}
\noindent
We present a definition of indefinite Kasparov modules, a generalisation 
of unbounded Kasparov
modules modelling non-symmetric and non-elliptic (e.g.\ hyperbolic) operators. 
Our main theorem shows that to each indefinite Kasparov module 
we can associate a pair of (genuine) Kasparov modules, and that this process is reversible. 
We present three examples of our framework:  the 
Dirac operator on a pseudo-Riemannian spin 
manifold (i.e.\ a manifold with an indefinite metric);
the harmonic oscillator; and the construction via the Kasparov product of an 
indefinite spectral triple from a family of spectral triples. This last construction corresponds 
to a foliation of a globally hyperbolic spacetime by spacelike hypersurfaces. 

\vspace{\baselineskip}\noindent

\end{abstract}


\parindent=0.0in
\setlength{\parskip}{3pt plus 3pt minus 1pt}

\section{Introduction}

Both Connes' noncommutative geometry \cite{Connes94} and Kasparov's 
$KK$-theory \cite{Kas80b,BJ83} deal with noncommutative 
generalisations of elliptic, self-adjoint differential operators. 
As such, these frameworks are particularly suited to describe 
{\em Riemannian} manifolds. In this article we aim to extend these 
frameworks to allow for non-elliptic and non-symmetric operators, 
and in particular (normally) hyperbolic operators. Our motivating 
example is the Dirac operator on a {\em pseudo-Riemannian} 
manifold, i.e.\ a manifold equipped with an \emph{indefinite} 
(but non-degenerate) metric. It is precisely this example 
that has inspired the terminology for the \emph{indefinite} 
Kasparov modules we introduce in \cref{defn:indef_Kasp_mod}.

Our definition of indefinite Kasparov modules is a 
generalisation of the usual definition of unbounded 
Kasparov modules \cite{BJ83}. One of our main 
goals is to make sure that this generalised definition 
still allows us to remain in touch with all the usual tools of 
$KK$-theory \cite{Kas80b}. If $(\A,E_B,\D)$ is an indefinite 
unbounded Kasparov module, we can construct from 
the (typically non-symmetric) operator $\D$ two symmetric operators given by
$$
\D_\pm := \Re\D \pm \Im\D = \frac12 (\D+\D^*) \mp \frac i2 (\D-\D^*) .
$$
We then want to make sure that these operators yield two 
unbounded Kasparov modules $(\A,E_B,\D_\pm)$, 
and the main challenge here is to prove self-adjointness for $\D_\pm$. 

This article continues in the spirit of our previous 
paper \cite{vdDPR13}, where we defined \emph{pseudo-Riemannian spectral triples} 
$(\A,\mH,\D)$ as a generalisation of spectral triples, 
and we showed that the operators $\D_\pm$ defined 
as above yield spectral triples.  Although the motivation 
for the present article is the same, there are nonetheless 
several significant differences. 

First, we work more 
generally with Kasparov modules instead of spectral triples. 
Second, while the definition of pseudo-Riemannian spectral 
triples requires assumptions on the \emph{second-order} 
operators $\D\D^*+\D^*\D$ and $\D^2-\D^{*2}$, the definition 
of indefinite Kasparov modules focuses more on \emph{first-order} 
operators (namely $\D$, $\D^*$, $\Re\D$, and $\Im\D$), 
which is more natural. Third, the definition of indefinite 
Kasparov modules has the advantage that it does not 
require any smoothness properties. And fourth, it allows to 
\emph{reverse} the procedure $\D \mapsto \D_\pm$, which 
means that we can characterise all pairs of unbounded Kasparov 
modules that can be obtained from an indefinite Kasparov module in this way. 

As mentioned above, the main technical challenge 
is to obtain self-adjointness for $\D_\pm$. In \cite{vdDPR13}, 
this is achieved by assuming that $\la\D\ra^2 := (\Re\D)^2+(\Im\D)^2$ 
is self-adjoint, and that the anti-commutator $\{\Re\D,\Im\D\}$ 
is `suitably bounded' relative to $\la\D\ra^2$. In this article, 
we prefer to avoid assumptions on the second-order operator 
$\la\D\ra^2$. Instead, we now impose the condition that the 
real and imaginary parts of $\D$ \emph{almost anti-commute}, 
which means that the anti-commutator $\{\Re\D,\Im\D\}$ is 
relatively bounded by $\Re\D$. A theorem of Kaad-Lesch 
\cite{KL12} (quoted in \cref{thm:sum-sa-reg})
then allows us to conclude 
that $\D_\pm$ are self-adjoint. 

Unfortunately, our main motivating 
example, namely the Dirac operator $\sD$ on a pseudo-Riemannian 
manifold, does not satisfy this condition. Indeed, although the 
anti-commutator $\{\Re\sD,\Im\sD\}$ is a first-order differential 
operator, it contains in general both spacelike derivatives and 
timelike derivatives, and thus it is not relatively bounded by $\Re\sD$ (nor by $\Im\sD$). 
In order to ensure that $\Re\sD$ and $\Im\sD$ almost anti-commute, 
we need the timelike part of $\{\Re\sD,\Im\sD\}$ to vanish identically, 
which places a restriction on the geometry of the pseudo-Riemannian manifold (see below). 
This asymmetry between the timelike and spacelike parts of $\{\Re\sD,\Im\sD\}$ is artificial, 
and indicates that it would be desirable to have a more general version of Kaad and Lesch'
theorem: we aim to return to this issue in a future work.

The layout of this article is as follows. 
In \cref{sec:bad-name} we first describe our approach to 
dealing with non-symmetric operators, where we 
emphasise the real and imaginary parts of the operator. 
Subsequently, we gather some results on almost (anti-)commuting 
operators which will be useful later on.

Next, we define indefinite Kasparov modules as well as 
pairs of Kasparov modules in \cref{sec:indef_KK}, and 
we prove our main theorem, which states that these 
definitions are equivalent. We continue in \cref{sec:indef_odd} 
by discussing the odd version of indefinite Kasparov modules. 
As for usual Kasparov modules, it is straightforward to turn an 
odd indefinite Kasparov module into an even one by `doubling it up'. 
We then prove that these odd modules are characterised by pairs of 
Kasparov modules for which the two operators are related via a certain unitary equivalence. 

In \cref{sec:egs} we discuss several examples. We start 
in \cref{sec:pseudo-Riem-mfd} with the main motivating 
example, namely the Dirac operator on a pseudo-Riemannian 
spin manifold. We show, under certain mild assumptions on the 
manifold, that this Dirac operator satisfies all but one condition 
in the definition of indefinite Kasparov modules. The condition 
that fails (as mentioned above) is the assumption that the real 
and imaginary parts of the Dirac operator almost anti-commute in the sense of Kaad-Lesch. 
We continue to show that this condition does hold for the 
case of Lorentzian manifolds with `parallel time'. This example indicates that further study is required
to obtain a more flexible formulation of indefinite Kasparov modules.

The second example (\cref{sec:harm-osc}) considers the 
harmonic oscillator in arbitrary dimensions. This example in 
particular shows that manifolds with indefinite metrics are 
not the only examples of our framework.

Finally, in \cref{sec:families} we discuss families of spectral 
triples (building upon work by Kaad and Lesch \cite{KL13}), 
and we show that one can naturally associate an indefinite 
Kasparov module to such families. Our work on families of 
spectral triples was initially motivated by the study of spacelike foliations 
of spacetime from the perspective of noncommutative geometry. 

{\bf Acknowledgments}
The first author acknowledges support from both the 
Australian National University and the University of Wollongong. The second author acknowledges
the support of the Australian Research Council. Both authors thank Magnus Goffeng and
Bram Mesland for insightful discussions, and the referee for identifying errors in a previous version. 
Both authors also thank the Hausdorff Research Institute for Mathematics (HIM) for their 
hospitality during the Hausdorff Trimester Program \emph{Non-commutative 
Geometry and its Applications} in 2014, where this work was first presented.

\section{Preliminaries on unbounded operators on Hilbert modules}
\label{sec:bad-name}
Let $B$ be a $\Z_2$-graded $C^*$-algebra. 
Recall that a $\Z_2$-graded Hilbert $B$-module $E$ is a vector space 
equipped with a $\Z_2$-graded right action $E\times B\to E$ and with a 
$B$-valued inner product $(\cdot|\cdot)\colon E\times E\to B$, such that 
$E$ is complete in the corresponding norm. 
The endomorphisms $\End_B(E)$ are the adjointable linear operators 
$E\to E$, and the set $\End^0_B(E)$ of compact endomorphisms is 
given by the closure of the finite rank operators. 
For a detailed introduction to Hilbert modules and $\Z_2$-gradings, 
we refer to \cite{Blackadar98,Lance95}. 
\subsection{Non-symmetric operators}
\label{sec:non-symm}
In this section, we describe our approach to dealing with 
non-symmetric operators, namely by studying the real and imaginary parts of such operators. 
Let us start with a useful lemma regarding the `combined graph norm' of 
two closed operators on the intersection of their domains.
\begin{lem}
\label{lem:complete_S_T}
Let $E$ be a right Hilbert $B$-module with inner product 
$(\cdot|\cdot)$. Let $S$ and $T$ be closed regular operators on 
$E$ such that $\Dom S\cap\Dom T$ is dense in $E$. Then 
$\Dom S\cap\Dom T$ is a right Hilbert $B$-module with the inner product
$$
(\phi|\psi)_{S,T} := (\phi|\psi) + (S\phi|S\psi) + (T\phi|T\psi) ,
$$
and the corresponding norm 
$\|\psi\|_{S,T}^2 = \|(\psi|\psi)_{S,T}\|_B$.
\end{lem}
\begin{proof}
We need to show that $\Dom S\cap\Dom T$ is complete in the 
norm $\|\cdot\|_{S,T}$. Since $S$ is closed, we know that 
$\Dom S$ is complete for the graph norm $\|\cdot\|_S$ corresponding to the inner product
$$
(\phi|\psi)_S := (\phi|\psi) + (S\phi|S\psi) ,
$$
and a similar statement holds for $\Dom T$. 
The inequalities 
$$
\frac12 (\psi|\psi)_S + \frac12 (\psi|\psi)_T \leq (\psi|\psi)_{S,T} \leq (\psi|\psi)_S + (\psi|\psi)_T
$$
show that convergence in the norm $\|\cdot\|_{S,T}$ is 
equivalent to convergence in both graph norms $\|\cdot\|_S$ 
and $\|\cdot\|_T$. Denote by $\mW_S$ (respectively $\mW_T$) 
the closure of $\Dom S\cap\Dom T$ in the norm $\|\cdot\|_S$ 
(respectively $\|\cdot\|_T$). Then the closure of $\Dom S\cap\Dom T$ 
in the norm $\|\cdot\|_{S,T}$ is contained in the intersection of $\mW_S$ 
and $\mW_T$. Since $\mW_S\subset\Dom S$ and $\mW_T\subset\Dom T$, 
this intersection $\mW_S\cap\mW_T$ is contained in, and hence equal to, 
$\Dom S\cap\Dom T$, so we conclude that $\Dom S\cap\Dom T$ is 
complete in the norm $\|\cdot\|_{S,T}$. 
\end{proof}
In what follows, we will consider a closed regular operator $\D$ 
on a right Hilbert $B$-module $E$, such that $\Dom\D\cap\Dom\D^*$ 
is dense in $E$. The above lemma then tells us that $\Dom\D\cap\Dom\D^*$ 
is a Hilbert $B$-module with the inner product $(\cdot|\cdot)_{\D,\D^*}$.
\begin{defn}
Let $\D$ be a closed regular operator on a Hilbert 
$B$-module $E$, such that $\Dom\D\cap\Dom\D^*$ is dense. 
We define the \emph{real and imaginary parts} of $\D$ by setting
\begin{align*}
\Re\D &:= \frac12 (\D+\D^*) , & \Im\D &:= -\frac i2 (\D-\D^*) ,
\end{align*}
on the initial domain $\Dom\D\cap\Dom\D^*$. Since these 
operators are densely defined and symmetric, they are closable, 
and we denote their closures by $\Re\D$ and $\Im\D$ as well. 
Furthermore, we define the \emph{`Wick rotations'} of $\D$ by 
\begin{align*}
\D_+ &:= \Re\D + \Im\D , & \D_- &:= \Re\D - \Im\D ,
\end{align*}
on the initial domain $\Dom\Re\D\cap\Dom\Im\D$. 
The term `Wick rotation' is borrowed from physics, and its use is motivated by \cref{prop:Wick_commuting_diagram}. 
\end{defn}
\begin{lem}
\label{lem:norms_D}
Let $\D$ be a closed regular operator on a right Hilbert 
$B$-module $E$, such that $\Dom\D\cap\Dom\D^*$ is dense in $E$. 
Then the norms $\|\cdot\|_{\D,\D^*}$, $\|\cdot\|_{\Re\D,\Im\D}$, 
and $\|\cdot\|_{\D_+,\D_-}$ (defined as in \cref{lem:complete_S_T}) are equivalent on $\Dom\D\cap\Dom\D^*$.
\end{lem}
\begin{proof}
An elementary calculation shows that we have the equalities 
\begin{align*}
(\phi|\psi)_{\Re\D,\Im\D} &= \frac12 (\phi|\psi) + \frac12 (\phi|\psi)_{\D,\D^*} , & (\phi|\psi)_{\D_+,\D_-} 
&= (\phi|\psi)_{\D,\D^*} ,
\end{align*}
from which it follows that the three norms $\|\cdot\|_{\D,\D^*}$, 
$\|\cdot\|_{\Re\D,\Im\D}$, and $\|\cdot\|_{\D_+,\D_-}$ are equivalent. 
\end{proof}
\begin{lem}
\label{lem:common_core}
Let $\D$ be a closed regular operator on a Hilbert 
$B$-module $E$, such that $\Dom\D\cap\Dom\D^*$ is dense. If $\Dom\D\cap\Dom\D^*$ is a core for both $\D$ and $\D^*$, then $\Dom\D\cap\Dom\D^* = \Dom\Re\D\cap\Dom\Im\D$, $\D = \Re\D+i\Im\D$, and $\D^* = \Re\D-i\Im\D$. 
\end{lem}
\begin{proof}
The operators $\Re\D$ and $\Im\D$ are initially defined on $\Dom\D\cap\Dom\D^*$, so the inclusion $\Dom\D\cap\Dom\D^* \subset \Dom\Re\D\cap\Dom\Im\D$ is obvious. Suppose that  $\Dom\D\cap\Dom\D^*$ is a core for both $\D$ and $\D^*$. Consider the closed operator $\til\D$  defined as the closure of $\Re\D + i \Im\D$ on the initial domain $\Dom\Re\D\cap\Dom\Im\D$. Obviously, $\til\D$ and $\D$ agree on $\Dom\D\cap\Dom\D^*$, and since this domain is a core for $\D$, it follows that $\til\D$ is an extension of $\D$, and we have $\til\D^* \subset \D^*$. However, on $\Dom\D\cap\Dom\D^*$ both $\til\D^*$ and $\D^*$ are given by $\Re\D-i\Im\D$, and since this domain is a core for $\D^*$, it follows that $\D^* \subset \til\D^*$. Hence $\D^* = \til\D^*$ and therefore $\D = \til\D$. By construction of $\til\D$ we have the domain inclusions
$$
\Dom\D\cap\Dom\D^* \subset \Dom\Re\D\cap\Dom\Im\D \subset \Dom\til\D\cap\Dom\til\D^* .
$$
Since we have shown that $\D=\til\D$, we conclude that these inclusions are equalities. 
\end{proof}
\begin{defn}
Let $\D_1$ and $\D_2$ be closed, regular and symmetric
operators on a Hilbert $B$-module $E$, such that 
$\Dom\D_1\cap\Dom\D_2$ is dense in $E$. 
We define the \emph{reverse Wick rotation} of the pair $(\D_1,\D_2)$ as the closure of
$$
\D := \frac12 (\D_1+\D_2) + \frac i2 (\D_1-\D_2) 
$$
on the initial domain $\Dom\D_1\cap\Dom\D_2$ (note that 
$\D$ is closable, because it is the sum of a symmetric and an 
anti-symmetric operator, which ensures that $\D^*$ is densely defined). 
\end{defn}
\begin{remark}
We emphasise that the reverse Wick rotation $\D'$ of the pair 
$(\D_2,\D_1)$ is \emph{not} equal to the reverse Wick rotation of 
$(\D_1,\D_2)$, but they are related to each other: $\D'$ is the 
closure of the restriction of $\D^*$ to $\Dom\D_1\cap\Dom\D_2$. 
In other words, $\D^*$ is a closed extension of the closure of $\D'$, 
and they are equal if and only if $\Dom\D_1\cap\Dom\D_2$ is a core for $\D^*$.
\end{remark}
\begin{lem}
\label{lem:norms-1-2}
Let $\D_1$ and $\D_2$ be closed, regular and symmetric operators on a 
Hilbert $B$-module $E$ such that $\Dom\D_1\cap\Dom\D_2$ is dense in $E$. 
Let $\D$ be the reverse Wick rotation of $(\D_1,\D_2)$. 
Then the norms
$\|\cdot\|_{\D,\D^*}$, $\|\cdot\|_{\Re\D,\Im\D}$, and $\|\cdot\|_{\D_1,\D_2}$ 
are all equivalent on $\Dom\D_1\cap\Dom\D_2$.
\end{lem}
\begin{proof}
Let us write $\E := \Dom\D_1\cap\Dom\D_2$. 
The operators $\D_1\pm\D_2$ are symmetric on $\E$, so the 
domain of $\D^*$ also contains $\E$ (and in particular $\Dom\D\cap\Dom\D^*$ is dense). 
For $\psi\in\E$ we can then write
$$
\D^* \psi = \frac12 (\D_1+\D_2) \psi - \frac i2 (\D_1-\D_2) \psi .
$$
Hence on the initial domain $\E$ we can write $\D_1 = \D_+$, 
$\D_2 = \D_-$, $\Re\D = \frac12(\D_1+\D_2)$, and 
$\Im\D = \frac12(\D_1-\D_2)$. From \cref{lem:norms_D} it then 
follows that the norms $\|\cdot\|_{\D,\D^*}$, $\|\cdot\|_{\Re\D,\Im\D}$, 
and $\|\cdot\|_{\D_1,\D_2}$  are equivalent on $\E$.
\end{proof}
\subsection{Almost (anti-)commuting operators}
\label{sec:almost-ac-op}
Almost (anti-)commuting operators were considered 
by Mesland in \cite{Mes14}, and later generalised by Kaad
and Lesch in \cite{KL12}, for the construction of the unbounded Kasparov product. 
Almost anti-commuting operators play an important role later in proving
that the Wick rotations of indefinite Kasparov modules are (genuine) Kasparov modules.
In this section, we recall the results from \cite{KL12}, and prove a few further consequences. 
\begin{defn}[see {\cite[Assumption 7.1]{KL12}}]
\label{defn:almost_(anti-)commute}
Let $S$ and $T$ be regular self-adjoint operators on a Hilbert $A$-module $E$ such that\\
1) there exists a submodule $\E\subset\Dom T$ which is a core for $T$, and\\
2) for each $\xi\in\E$ and for all $\mu\in\R\backslash\{0\}$ we have the inclusions
\begin{align*}
(S-i\mu)^{-1} \xi &\in \Dom S\cap\Dom T\quad\mbox{and}\quad  T(S-i\mu)^{-1}\xi \in \Dom S .
\end{align*}
The pair $(S,T)$ is called an \emph{almost commuting pair} if in addition\\
3) The map $[S,T] (S-i\mu)^{-1} \colon \E \to E$ extends to a bounded operator 
in $\End_A(E)$ for all $\mu\in\R\backslash\{0\}$.

Similarly, the pair $(S,T)$ is called an \emph{almost anti-commuting pair} if instead of 3) we have\\
3') The map $\{S,T\} (S-i\mu)^{-1} \colon \E \to E$ extends to a 
bounded operator in $\End_A(E)$ for all $\mu\in\R\backslash\{0\}$.

These conditions are often summarised by simply saying that 
$[S,T](S-i\mu)^{-1}$ (or $\{S,T\}(S-i\mu)^{-1}$) is well-defined and bounded. 
\end{defn}
\begin{lem}
\label{lem:pair_S_ess-sa}
Let $(S,T)$ be a pair of regular self-adjoint operators on a Hilbert module $E$
satisfying $1)$ and $2)$ of Definition \ref{defn:almost_(anti-)commute}. 
Then $S$ is essentially self-adjoint on $\Dom S\cap\Dom T$.
\end{lem}
\begin{proof}
By assumption we have $(S\pm i)^{-1}(\xi) \in \Dom S\cap\Dom T$ for all $\xi\in \E$, where $\E$ is dense in $E$.
Since $S$ is self-adjoint, the operator $(S\pm i)^{-1}$ is bounded and has range $\Dom S$, which is dense in $E$. Hence $(S\pm i)^{-1}\E$ is also dense in $E$, from which it follows that $\Dom S\cap\Dom T$ is dense in $E$, so the operator $S|_{\Dom S\cap\Dom T}$ is symmetric and densely defined on $\Dom S\cap\Dom T$. Furthermore, the image of $(S\pm i)|_{\Dom S\cap\Dom T}$ contains $\E$ and is therefore also dense, which implies that $S|_{\Dom S\cap\Dom T}$ is essentially self-adjoint. 
\end{proof}
Given two regular self-adjoint operators $S$ and $T$ on a Hilbert $A$-module $E$, we consider two new operators on $E\oplus E$ given by
\begin{align*}
\til S &:= \mattwo{0}{iS}{-iS}{0} , & \til T &:= \mattwo{0}{T}{T}{0} ,
\end{align*}
with domains $\Dom\til S = (\Dom S)^{\oplus2}$ and $\Dom\til T = (\Dom T)^{\oplus2}$. 
One easily calculates that
\begin{align*}
\{\til S,\til T\} &= i\mattwo{[S,T]}{0}{0}{-[S,T]} , & [\til S,\til T] &= i\mattwo{\{S,T\}}{0}{0}{-\{S,T\}} ,
\end{align*}
whenever these operators are defined. Hence this `doubling trick' allows us to easily switch between almost commuting and anti-commuting operators. 
\begin{lem}
\label{lem:comm_anti-comm}
Let $S$ and $T$ be regular self-adjoint operators on a Hilbert $A$-module $E$, and let $\til S$ and $\til T$ be given as above. Then the following statements hold:\\
1) if $(S,T)$ is an almost commuting pair, then $(\til S,\til T)$ is an almost anti-commuting pair;\\
2) if $(S,T)$ is an almost anti-commuting pair, then $(\til S,\til T)$ is an almost commuting pair.
\end{lem}
\begin{proof}
We only prove the first statement, as the second statement is similar. So suppose that the operator $[S,T](S-i\mu)^{-1}\colon\E\to E$ is well-defined and bounded. Consider the submodule $\til\E=\E\oplus\E$ of the Hilbert module $\til E=E\oplus E$. 
An explicit calculation shows that we can rewrite
\begin{align*}
(\til S - i \mu)^{-1} &= \mattwo{-i\mu}{iS}{-iS}{-i\mu}^{-1} = \mattwo{(S-i\mu)^{-1}}{0}{0}{(S-i\mu)^{-1}} \mattwo{(S-i\mu)}{0}{0}{(S-i\mu)} \mattwo{-i\mu}{iS}{-iS}{-i\mu}^{-1} \\
&= \mattwo{(S-i\mu)^{-1}}{0}{0}{(S-i\mu)^{-1}} \mattwo{i\mu(S+i\mu)^{-1}}{iS(S+i\mu)^{-1}}{-iS(S+i\mu)^{-1}}{i\mu(S+i\mu)^{-1}} .
\end{align*}
The second matrix on the second line is bounded, and it maps $\til\E$ to $\Dom\til T$ (by assumption, $(S+i\mu)^{-1}$ maps $\E$ to $\Dom S\cap\Dom T$, and $S(S+i\mu)^{-1} = 1 - i\mu (S+i\mu)^{-1}$ maps $\E$ to $\Dom T$). 
Since the submodule $\E$ in \cref {defn:almost_(anti-)commute} can always be replaced by $\Dom T$ (see \cite[Proposition 7.3]{KL12}), this shows that $\til\E$ satisfies conditions 1) and 2). 
Furthermore, the operator
\begin{align*}
\{\til S,\til T\} (\til S - i \mu)^{-1} &= i\mattwo{[S,T](S-i\mu)^{-1}}{0}{0}{-[S,T](S-i\mu)^{-1}}\mattwo{i\mu(S+i\mu)^{-1}}{iS(S+i\mu)^{-1}}{-iS(S+i\mu)^{-1}}{i\mu(S+i\mu)^{-1}} .
\end{align*}
is then well-defined and bounded on $\til\E$. 
\end{proof}
\begin{thm}[{\cite[Theorem 7.10]{KL12}}]
\label{thm:sum-sa-reg}
Let $(S,T)$ be an almost commuting pair of regular self-adjoint operators on $E$. Then the operator
$$
\D := \mattwo{0}{S+iT}{S-iT}{0}
$$
with domain $\Dom\D := \big(\Dom S\cap\Dom T\big)^{\oplus2}$ is self-adjoint and regular.
\end{thm}
Combining Theorem \ref{thm:sum-sa-reg} with  \cref{lem:comm_anti-comm}, 
we obtain a variant of Kaad and Lesch' result.
\begin{coro}
\label{coro:anti-sum-sa-reg}
Let $(S,T)$ be an almost \emph{anti}-commuting pair of regular self-adjoint operators on $E$. Then the operators $S+T$ and $S-T$ with domain $\Dom S\pm T = \Dom S\cap\Dom T$ are regular and self-adjoint. 
\end{coro}
From the assumption that $(S,T)$ is an almost commuting pair, it does not follow that $S\pm T$ is self-adjoint on $\Dom S\cap\Dom T$ (the obvious counter-example is $S=\mp T$). However, it does follow that $S\pm T$ is \emph{essentially} self-adjoint on $\Dom S\cap\Dom T$.
\begin{prop}
\label{prop:sum-sa}
Let $(S,T)$ be an almost commuting pair of regular self-adjoint operators on $E$. Then the operator $S+T$ is essentially self-adjoint on  $\Dom S\cap\Dom T$. 
\end{prop}
\begin{proof}
The statement follows from a straightforward adaptation of the proof of \cite[Proposition 7.7]{KL12}, which we include here for completeness. We know that $S+T$ is symmetric on $\Dom S\cap\Dom T$, so it suffices to prove that $\Dom(S+T)^*\subset\Dom\bar{(S+T)}$. Let $\xi\in\Dom(S+T)^*$, and define the sequence
$$
\xi_n := \Big(-\frac inS+1\Big)^{-1} \xi \in \Dom S ,
$$
which converges in norm to $\xi$. For $\eta\in\Dom T$, we can calculate
\begin{align*}
\la\xi_n,T\eta\ra &= \Big\la \xi,\Big(\frac in S+1\Big)^{-1} T\eta \Big\ra 
= \Big\la \xi , T \Big(\frac in S+1\Big)^{-1} \eta \Big\ra 
- \Big\la \xi , \Big(\frac in S+1\Big)^{-1} \Big[\frac in S,T\Big] \Big(\frac in S+1\Big)^{-1} \eta \Big\ra \\
&= \Big\la \xi , (S+T) \Big(\frac in S+1\Big)^{-1} \eta \Big\ra 
- \Big\la \xi , S \Big(\frac in S+1\Big)^{-1} \eta \Big\ra 
- \la \xi , R_n \eta \ra \\
&= \Big\la \Big(-\frac in S+1\Big)^{-1} (S+T)^* \xi , \eta \Big\ra 
- \Big\la S \Big(-\frac in S+1\Big)^{-1} \xi , \eta \Big\ra - \la R_n^* \xi , \eta \ra ,
\end{align*}
where $R_n := (\frac in S+1)^{-1} [\frac in S,T] (\frac in S+1)^{-1}$ is defined as in \cite[Lemma 7.4]{KL12}. This proves that $\xi_n$ is in the domain of $T^*=T$, and that
$$
T\xi_n = \Big(-\frac in S+1\Big)^{-1} (S+T)^* \xi - S \xi_n - R_n^* \xi .
$$
In \cite[Lemma 7.4]{KL12} it is shown that $R_n\to 0$ strongly, and hence
$$
(S+T)\xi_n = \Big(-\frac in S+1\Big)^{-1} (S+T)^* \xi - R_n^* \xi
$$
converges in norm to $(S+T)^*\xi$, which in particular means that $\xi = \lim_{n\to\infty}\xi_n \in \Dom\bar{(S+T)}$. 
\end{proof}

\section{Indefinite Kasparov modules}
\label{sec:indef_KK}
We are now ready to present our framework of indefinite Kasparov modules, after
which we define pairs of Kasparov modules, and show that these definitions are equivalent. 
\begin{defn}
\label{defn:indef_Kasp_mod}
Given (separable) $\Z_2$-graded $C^*$-algebras $A$ and $B$, an \emph{indefinite} unbounded Kasparov 
$A$-$B$-module $(\A,{}_\pi E_B,\D)$ is given by\vspace{-8pt}
\begin{itemize}
\item a $\Z_2$-graded, countably generated, right Hilbert $B$-module $E$;\vspace{-6pt}
\item a $\Z_2$-graded $*$-homomorphism $\pi\colon A\to\End_B(E)$;\vspace{-6pt}
\item a separable dense $*$-subalgebra $\A\subset A$;\vspace{-6pt}
\item a closed, regular, odd operator $\D\colon\Dom\D\subset E\to E$ such that \\
1) there exists a linear subspace $\E\subset\Dom\D\cap\Dom\D^*$ which is dense with respect to 
$\|\cdot\|_{\D,\D^*}$, and which is a core for both $\D$ and $\D^*$;\\
2) the operators $\Re\D$ and $\Im\D$ are regular and essentially self-adjoint on $\E$; \\
3) the pair $(\Re\D,\Im\D)$ is an \emph{almost anti-commuting pair}; \\
4) we have the inclusion 
$\pi(\A)\cdot\E\subset\Dom\D\cap\Dom\D^*$, and the graded commutators 
$[\D,\pi(a)]_\pm$ and $[\D^*,\pi(a)]_\pm$ are bounded on $\E$ for each $a\in\A$;\\
5) the map $\pi(a)\circ\iota\colon \Dom\D\cap\Dom\D^*\into E\to E$ is 
compact for each $a\in A$, where $\iota\colon \Dom\D\cap\Dom\D^*\into E$ 
denotes the natural inclusion map, and $\Dom\D\cap\Dom\D^*$ is considered as a 
Hilbert $B$-module with the inner product $(\cdot|\cdot)_{\D,\D^*}$.
\end{itemize}
\vspace{-8pt}
If no confusion arises, we will often write $(\A,E_B,\D)$ instead of $(\A,{}_\pi E_B,\D)$. 
If $B=\C$ and $A$ is trivially graded, we will write 
$E=\mH$ and refer to $(\A,\mH,\D)$ as an even indefinite spectral triple over $A$. 
\end{defn}
\begin{remark}
If $\D$ is self-adjoint, this is just the usual definition of an 
unbounded Kasparov $A$-$B$-module (or spectral
triple if $B=\C$). 
In this case, note that assumption 5) is equivalent to the more commonly used 
assumption that the resolvent of $\D$ is \emph{locally compact} 
(which means that the operator $\pi(a) (1+\D^2)^{-1/2}$ is compact for each $a\in A$). 

\label{rem:indef_Kasp_mod_dom_eq}
Property 1) and \cref{lem:common_core} imply that we have the equality $\Dom\D\cap\Dom\D^* = \Dom\Re\D\cap\Dom\Im\D$, which we will use repeatedly. 
\end{remark}
\begin{defn}
\label{defn:unit-equiv}
Two indefinite unbounded Kasparov $A$-$B$-modules 
$(\A,E_1,\D_1)$ and $(\A,E_2,\D_2)$ are called 
\emph{unitarily equivalent} if there exists an even unitary 
$U\colon E_1\to E_2$ such that $\D_2 = U\D_1U^*$ and for all $a\in\A$ we have
$\pi_2(a) = U\pi_1(a)U^*$, where $\pi_i\colon A\to\End_B(E_i)$ 
denotes the left action of $A$ (for $i=1,2$). 
\end{defn}
Next, we will show that the linear subspace $\E$ in \cref{defn:indef_Kasp_mod} can always be replaced by $\Dom\D\cap\Dom\D^*$. The trickiest part turns out to be condition 4), for which we prove a separate lemma first. 
\begin{lem}[cf.\ {\cite[Proposition 2.1]{FMR14}}]
\label{lem:dense_bdd_comm}
Let $\D$ be a closed regular operator on a Hilbert $B$-module $E$ such that 
$\Dom\D\cap\Dom\D^*$ is dense. Let $\E\subset\Dom\D\cap\Dom\D^*$ be 
dense with respect to the norm $\|\cdot\|_{\D,\D^*}$, and let 
$\A\subset\End_B(E)$ be a $*$-subalgebra. 
Suppose $\A\cdot\E\subset\Dom\D\cap\Dom\D^*$, 
and for each $a\in\A$ the operators $[\D,a]$ and $[\D^*,a]$ are bounded on $\E$. 
Then $\A$ also preserves $\Dom\D\cap\Dom\D^*$, and $[\D,a]$ and $[\D^*,a]$, 
initially defined on $\Dom\D\cap\Dom\D^*$, extend to bounded endomorphisms on $E$, for all $a\in\A$. 
\end{lem}
\begin{proof}
The proof is a straightforward adaptation of \cite[Proposition 2.1]{FMR14}, which proves the statement for the case of self-adjoint operators on a Hilbert space. For completeness we will work out the details here. 

Let $\psi\in\Dom\D\cap\Dom\D^*$. By assumption there exists a sequence $\psi_n\in\E$ such that $\psi_n\to\psi$ in the norm $\|\cdot\|_{\D,\D^*}$, which is equivalent to $\psi_n\to\psi$, $\D\psi_n\to\D\psi$, and $\D^*\psi_n\to\D^*\psi$, in the usual norm. The sequence $\D a\psi_n$ is Cauchy (in the usual norm), since
\begin{align*}
\|\D a\psi_n - \D a\psi_m\| = \|a\D\psi_n - a\D\psi_m + [\D,a]\psi_n - [\D,a]\psi_m\| \leq \|a\| \|\D\psi_n-\D\psi_m\| + \|[\D,a]\| \|\psi_n-\psi_m\| ,
\end{align*}
and similarly $\D^*a\psi_n$ is also Cauchy. Hence the sequence $a\psi_n\in\Dom\D\cap\Dom\D^*$ is Cauchy in the norm $\|\cdot\|_{\D,\D^*}$, so there exists a $\xi\in\Dom\D\cap\Dom\D^*$ such that $a\psi_n\to\xi$ in the norm $\|\cdot\|_{\D,\D^*}$. But this implies that $a\psi_n\to\xi$ in the usual norm, and since we already know that $a\psi_n\to a\psi$ in the usual norm, we conclude that $\xi = a\psi$, and hence $a\psi \in \Dom\D\cap\Dom\D^*$. Thus we have shown that $a$ preserves $\Dom\D\cap\Dom\D^*$. 

To conclude that $[\D,a]$ (and similarly $[\D^*,a]$), initially defined on $\Dom\D\cap\Dom\D^*$, extends to a bounded endomorphism, it suffices to show that its adjoint is densely defined, since then it is closable, 
and $\bar{[\D,a]} \supset \bar{[\D,a]|_\E}$, which is everywhere defined and bounded. For $\psi\in\Dom\D$ and $\eta\in\E$, we have
$$
([\D,a]\psi| \eta) = ( \D a\psi|\eta)  - (a\D\psi|\eta) = (\psi|a^*\D^*\eta) - (\psi|\D^*a^*\eta) = (\psi|-[\D^*,a^*]\eta) ,
$$
which is well-defined because $a^*\in\A$ maps $\E$ to 
$\Dom\D\cap\Dom\D^*$. Hence the domain of $[\D,a]^*$ 
contains the dense subset $\E$, which implies that $[\D,a]$ is closable. 
The same argument applies to $[\D^*,a]$. 
\end{proof}
\begin{prop}
\label{prop:complete_E}
If $(\A,E_B,\D)$ is an indefinite unbounded Kasparov $A$-$B$-module, then the subset $\E$ in \cref{defn:indef_Kasp_mod} can be replaced by $\Dom\D\cap\Dom\D^*$. 
\end{prop}
\begin{proof}
If $\E\subset\Dom\D\cap\Dom\D^*$ is a core for $\D$ and $\D^*$, then so is 
$\Dom\D\cap\Dom\D^*$. By construction,  
$\Dom\D\cap\Dom\D^*$ is contained in the domains of 
$\Re\D$ and $\Im\D$, so the operators $\Re\D$ and 
$\Im\D$ are also essentially self-adjoint on $\Dom\D\cap\Dom\D^*$. 
Using \cref{lem:dense_bdd_comm} then concludes the proof. 
\end{proof}
\subsection{Pairs of Kasparov modules}
\begin{defn}
\label{defn:even_pair}
We say $(\A,E_B,\D_1,\D_2)$ is a 
\emph{pair of unbounded Kasparov $A$-$B$-modules} if $(\A,E_B,\D_1)$ and 
$(\A,E_B,\D_2)$ are unbounded Kasparov $A$-$B$-modules such that:\\
1) there exists a linear subspace $\E\subset\Dom\D_1\cap\Dom\D_2$ which is a common core for $\D_1$ and $\D_2$;\\
2) the operators $\D_1+\D_2$ and $\D_1-\D_2$ are regular and essentially self-adjoint on $\E$; \\
3) the pair $(\D_1+\D_2,\D_1-\D_2)$ is an almost anti-commuting pair. \\
If $B=\C$ and $A$ is trivially graded, we will write 
$E=\mH$ and refer to $(\A,\mH,\D_1,\D_2)$ as an even pair of spectral triples over $A$. 
\end{defn}
\begin{remark}
\label{remark:even_pair}
Since $(\D_1+\D_2,\D_1-\D_2)$ is 
an almost anti-commuting pair, it follows from 
\cref{coro:anti-sum-sa-reg} that in fact $\Dom\D_1 = \Dom\D_2$.  
Similarly to \cref{prop:complete_E}, we can then replace $\E$ by $\Dom\D_1=\Dom\D_2$. 
\end{remark}
\begin{prop}[Wick rotation]
\label{prop:indef-wick}
Let $(\A,E_B,\D)$ be an indefinite unbounded Kasparov $A$-$B$-module. 
Then the Wick rotations $\D_+$ and $\D_-$ form a pair of unbounded Kasparov 
$A$-$B$-modules $(\A,E_B,\D_+,\D_-)$. 
\end{prop}
\begin{proof}
By assumption, the operators $\Re\D$ and $\Im\D$ are essentially self-adjoint on $\Dom\Re\D\cap\Dom\Im\D$, and they form an 
almost anti-commuting pair $(\Re\D,\Im\D)$. 
By construction, we have the domain inclusions
$$
\Dom\Re\D\cap\Dom\Im\D \subset \Dom\D_+\cap\Dom\D_- \subset \Dom(\D_++\D_-)\cap\Dom(\D_+-\D_-) .
$$
Since $\D_++\D_-$ and $\D_+-\D_-$ are symmetric extensions of the essentially self-adjoint operators $2\Re\D$ and $2\Im\D$ (respectively), we must have $\D_++\D_- = 2\Re\D$ and 
$\D_+-\D_- = 2\Im\D$, which implies that the above domain inclusions are in fact equalities. This also proves properties 2) and 3). 
By \cref{coro:anti-sum-sa-reg} 
it then follows that $\D_\pm = \Re\D\pm\Im\D$ are self-adjoint on the domain 
$\E := \Dom\D_\pm = \Dom\Re\D\cap\Dom\Im\D$, which shows property 1). 

To complete the proof that $\D_\pm$ yield unbounded Kasparov modules, 
first observe that $[\Re\D,a]$ and $[\Im\D,a]$ are bounded on 
$\Dom\D_\pm = \Dom\Re\D\cap\Dom\Im\D$, 
and hence it follows that $[\Re\D\pm\Im\D,a]$, 
initially defined on $\Dom\Re\D\cap\Dom\Im\D$, extend to bounded endomorphisms on $E$. 

Finally, we know from \cref{lem:common_core} that $\Dom\D_\pm=\Dom\Re\D\cap\Dom\Im\D$ is equal to $\Dom\D\cap\Dom\D^*$ (with the same norm-topology), and by assumption the map $\pi(a)\circ\iota\colon\Dom\D\cap\Dom\D^*\to E$ is compact. 
Thus the Wick rotations $(\A,E_B,\D_\pm)$ are indeed unbounded Kasparov modules. 
\end{proof}
\begin{prop}[reverse Wick rotation]
\label{prop:pair_to_indef}
Let $(\A,E_B,\D_1,\D_2)$ be a pair of unbounded Kasparov $A$-$B$-modules, 
and let $\D$ be the reverse Wick rotation of $(\D_1,\D_2)$. 
Then $(\A,E_B,\D)$ is an indefinite unbounded Kasparov $A$-$B$-module. 
\end{prop}
\begin{proof}
As mentioned in \cref{remark:even_pair}, we can pick $\E = \Dom\D_1 = \Dom\D_2$. 
By construction, we have the domain inclusions
$$
\E \subset \Dom\D\cap\Dom\D^* \subset \Dom\Re\D\cap\Dom\Im\D \subset \Dom(\Re\D+\Im\D)\cap\Dom(\Re\D-\Im\D) .
$$
The operators $\D_+=\Re\D+\Im\D$ and $\D_-=\Re\D-\Im\D$ are symmetric extensions of the self-adjoint operators $\D_1$ and $\D_2$, and hence $\D_1 = \D_+$ and $\D_2=\D_-$. This implies that the above domain inclusions are in fact equalities. By definition, $\E$ is a core for the reverse Wick rotation $\D$. 
On this domain we can write
\begin{align*}
\Re\D \,\psi &= \frac12 (\D_1+\D_2) \psi , & \Im\D\, \psi &= \frac12 (\D_1-\D_2) \psi .
\end{align*}
Thus by assumption the operators $\Re\D$ and $\Im\D$ are 
essentially self-adjoint on $\E$, and they form an almost anti-commuting pair $(\Re\D,\Im\D)$. Since 
$\D_1$ and $\D_2$ have bounded commutators 
with  $\A$, it follows 
immediately that $\Re\D$ and $\Im\D$ also have 
bounded commutators with $\A$. We observe 
that the identity map 
$(\E,\|\cdot\|_{\D_1,\D_2})\to(\Dom\D_1,\|\cdot\|_{\D_1})$ 
is continuous, because the graph norm of $\D_1$ is 
bounded by the norm $\|\cdot\|_{\D_1,\D_2}$ on 
$\E$ (and similarly for $\D_2$).  
Since $\D_1$ (or $\D_2$) has locally compact resolvent, 
it then follows that the map $\pi(a)\circ\iota\colon\E\into E\to E$ 
is compact for each $a\in\A$.
Finally, it remains to show that $\E$ is a core for $\D^*$. Since $(\Re\D,\Im\D)$ is an almost anti-commuting pair, we can use the `doubling trick' and apply \cref{prop:sum-sa} to conclude that 
$$
\mattwo{0}{i(\Re\D-i\Im\D)}{-i(\Re\D+i\Im\D)}{0}
$$
is essentially self-adjoint on $\E\oplus\E$. Thus $\D^* = (\Re\D+i\Im\D)^*$ equals the closure of $\Re\D-i\Im\D$ on $\E$, so $\E$ is indeed a core for $\D^*$. 
\end{proof}
\begin{remark}
Let $(\A,E_B,\D)$ be an indefinite unbounded 
Kasparov module with Wick rotations $\D_+$ and $\D_-$. 
Denote by $\til\D$ the reverse Wick rotation of $(\D_+,\D_-)$. By construction we then have the domain inclusions
$$
\Dom\D\cap\Dom\D^* \subset \Dom\Re\D\cap\Dom\Im\D \subset \Dom\D_+\cap\Dom\D_- \subset \Dom\til\D\cap\Dom\til\D^*
$$
As in the proof of \cref{lem:common_core}, the assumption that $\Dom\D\cap\Dom\D^*$ is a core for both $\D$ and $\D^*$ implies that $\til\D = \D$.
Thus, 
this assumption ensures that our procedure of Wick rotation is reversible. 

We can consider unitary equivalences of indefinite 
Kasparov modules or pairs of Kasparov modules as in \cref{defn:unit-equiv}, 
and one easily sees that Wick rotations and reverse Wick rotations respect such unitary equivalences. 
\end{remark}
Combining these observations with \cref{prop:indef-wick,prop:pair_to_indef}, 
we can summarise our results as follows:
\begin{thm}
\label{thm:even_indef_pair_bij}
The procedure of (reverse) Wick rotation implements a bijection between
 indefinite unbounded Kasparov $A$-$B$-modules $(\A,E_B,\D)$ and 
 pairs of unbounded Kasparov $A$-$B$-modules $(\A,E_B,\D_1,\D_2)$. 
This bijection also descends to the corresponding unitary equivalence classes.
\end{thm}
\subsection{Odd indefinite Kasparov modules}
\label{sec:indef_odd}
We introduce an odd version of indefinite Kasparov modules, 
where all $\Z_2$-gradings are trivial, and the operator 
$\D$ is (of course) no longer assumed to be odd. 
\begin{defn}
\label{defn:odd_indef_Kasp_mod}
Given  trivially graded $C^*$-algebras $A$ and $B$, an 
\emph{odd indefinite} unbounded Kasparov $A$-$B$-module $(\A,E_B,\D)$ is given by\vspace{-8pt}
\begin{itemize}
\item a  trivially graded, countably generated, right Hilbert $B$-module $E$;\vspace{-6pt}
\item a $*$-homomorphism $\pi\colon A\to\End_B(E)$;\vspace{-6pt}
\item a separable dense $*$-subalgebra $\A\subset A$;\vspace{-6pt}
\item a closed, regular operator $\D\colon\Dom\D\subset E\to E$ such that \\
1)  there exists a linear subspace $\E\subset\Dom\D\cap\Dom\D^*$ which is dense 
in the norm $\|\cdot\|_{\D,\D^*}$ and which is a core for both $\D$ and $\D^*$;\\
2) the operators $\Re\D$ and $\Im\D$ are regular and essentially self-adjoint on $\E$; \\
3) the pair $(\Re\D,\Im\D)$ is an \emph{almost commuting pair}; \\
4) we have the inclusion $\pi(\A)\cdot\E\subset\Dom\D\cap\Dom\D^*$, 
and the commutators $[\D,\pi(a)]$ and $[\D^*,\pi(a)]$ 
are bounded on $\E$ for each $a\in\A$;\\
5) the map $\pi(a)\circ\iota\colon \Dom\D\cap\Dom\D^*\into E\to E$ is 
compact for each $a\in A$, where $\iota\colon \Dom\D\cap\Dom\D^*\into E$ 
denotes the natural inclusion map, and $\Dom\D\cap\Dom\D^*$ is considered as a 
Hilbert module with the inner product $(\cdot|\cdot)_{\D,\D^*}$.
\end{itemize}
If $B=\C$, we will write $E=\mH$ and refer to $(\A,\mH,\D)$ 
as an odd indefinite spectral triple over $A$. 
If $\D$ is self-adjoint, we recover the definition of an odd unbounded Kasparov module (or odd spectral triple). 
\end{defn}
\begin{remark}
\label{remark:odd_indef_Kasp_mod}
We emphasise that, in the odd case, the pair $(\Re\D,\Im\D)$ is assumed to almost 
\emph{commute} (instead of almost \emph{anti}-commute). This assumption
can be reinterpreted as saying that the commutator $[\D,\D^*]$ is relatively bounded 
by the sum $\D+\D^*$; in this sense $\D$ is \emph{`almost normal'}.  

It follows from 3) and 
\cref{thm:sum-sa-reg} 
that in fact we have $\Dom\D=\Dom\D^*$, and 
as in \cref{prop:complete_E} we can then always replace $\E$ by $\Dom\D$. 
\end{remark}
Given an odd indefinite unbounded Kasparov module 
$(\A,E_B,\D)$, we can again consider its Wick rotations 
\begin{align*}
\D_+ &:= \Re\D + \Im\D , &  \D_- &:= \Re\D - \Im\D ,
\end{align*}
on the initial domain $\Dom\Re\D\cap\Dom\Im\D$. The following example shows that these 
Wick rotations are not as well-behaved as in the $\Z_2$-graded case.
\begin{example}
Let $(\A,E_B,\D)$ be an odd unbounded Kasparov module, 
and consider the operator $\til\D:=(1+i)\D$. Then $(\A,E_B,\til\D)$ 
is an odd indefinite unbounded Kasparov module, and its 
Wick rotations are $\til\D_+ = 2\D$ and $\til\D_- = 0$. 
The problematic one is obviously $\til\D_-$, as it is not closed on 
$\Dom\D$, and it does not have locally compact resolvent. 
\end{example}
Hence the assumptions of an odd indefinite unbounded 
Kasparov module do not imply that the Wick rotations 
yield odd unbounded Kasparov modules. However, 
by \cref{prop:sum-sa} we do know that $\D_+$ and $\D_-$ 
are essentially self-adjoint, and we will denote their 
self-adjoint closures by $\D_+$ and $\D_-$ as well.

Given an odd unbounded Kasparov module $(\A,E_B,\D)$, it is straightforward to 
construct an (even) unbounded Kasparov module $(\A,\til E_B,\til\D)$ by defining the 
$\Z_2$-graded Hilbert module $\til E := E\oplus E$ (where the first 
summand is considered even and the second summand odd) and the odd operator
$$
\til\D := \mattwo{0}{\D}{\D}{0} .
$$
The following theorem gives a similar `doubling trick' for the indefinite case. 
\begin{thm}
\label{thm:odd_indef-even_pair}
Given trivially graded $C^*$-algebras $A$ and $B$, 
let $E_B$ be a trivially graded, countably generated right Hilbert $B$-module with a 
$*$-homomorphism $\pi\colon A\to\End_B(E)$, 
let $\A\subset A$ be a separable dense $*$-subalgebra, 
and let $\D\colon\Dom\D\to E$ be a closed, regular operator. 
Consider (the closures of) the operators
\begin{align*}
\D_+ &:= \Re\D + \Im\D , & \D_- &:= \Re\D - \Im\D , & \til\D &:= \mattwo{0}{\D_+}{\D_-}{0} , 
& \til\D_+ &:= \mattwo{0}{\D^*}{\D}{0} , & \til\D_- &:= \mattwo{0}{\D}{\D^*}{0} .
\end{align*}
Then the following are equivalent:\\
1) $(\A,E_B,\D)$ is an odd indefinite unbounded Kasparov $A$-$B$-module;\\
2) $(\A,(E\oplus E)_B,\til\D)$ is an indefinite unbounded Kasparov $A$-$B$-module;\\
3) $(\A,(E\oplus E)_B,\til\D_+,\til\D_-)$ is a pair of unbounded Kasparov $A$-$B$-modules.
\end{thm}
\begin{proof}
One easily sees that the reverse Wick rotation of 
$(\til\D_+,\til\D_-)$ equals $\til\D$, and the equivalence of 2) and 3) then follows from \cref{thm:even_indef_pair_bij}. Hence it suffices to prove the equivalence of 1) and 3). 
\begin{description}
\item[1)$\Rightarrow$3):] Let $(\A,E_B,\D)$ be an odd indefinite unbounded 
Kasparov $A$-$B$-module. 
From \cref{remark:odd_indef_Kasp_mod} we have the equality  $\Dom\D = \Dom\D^*$, and from \cref{lem:common_core} we then know that $\Dom\D = \Dom\Re\D\cap\Dom\Im\D$, and we can write $\D = \Re\D + i \Im\D$ and $\D^* = \Re\D-i\Im\D$.
Thus the operators
\begin{align*}
\til\D_+ &= \mattwo{0}{\D^*}{\D}{0} 
 , & \til\D_- &= \mattwo{0}{\D}{\D^*}{0} 
\end{align*}
are self-adjoint on $(\Dom\Re\D\cap\Dom\Im\D)^{\oplus2}$. 
For all $a\in\A$, we know that $[\D,a]$ and $[\D^*,a]$ are bounded, 
and therefore $[\til\D_+,a]$ and $[\til\D_-,a]$ are also bounded. 
Furthermore, the inclusion of the domain $(\Dom\Re\D\cap\Dom\Im\D)^{\oplus2}$ 
in $E\oplus E$ is locally compact, because the inclusion 
$\Dom\Re\D\cap\Dom\Im\D = \Dom\D \hookrightarrow E$ is 
locally compact by assumption. Thus both $(\A,(E\oplus E)_B,\til\D_\pm)$ 
are unbounded Kasparov $A$-$B$-modules. 

The operators $\til\D_+\pm\til\D_-$ are essentially self-adjoint on 
$(\Dom\Re\D\cap\Dom\Im\D)^{\oplus2}$, because $\Re\D$ and 
$\Im\D$ are essentially self-adjoint on $\Dom\Re\D\cap\Dom\Im\D$. 
Since $(\Re\D,\Im\D)$ is an almost commuting pair, it follows from \cref{lem:comm_anti-comm} that $(\til\D_++\til\D_-,\til\D_+-\til\D_-)$ is an almost anti-commuting pair.
Thus $(\A,(E\oplus E)_B,\til\D_+,\til\D_-)$ is indeed a pair of 
unbounded Kasparov $A$-$B$-modules. 

\item[3)$\Rightarrow$1):] Suppose $(\A,(E\oplus E)_B,\til\D_+,\til\D_-)$ 
is a pair of unbounded Kasparov $A$-$B$-modules. 
The property $\Dom\til\D_+=\Dom\til\D_-$ (see \cref{remark:even_pair}) then implies that 
$\Dom\D=\Dom\D^*$. Since $\til\D_+\pm\til\D_-$ are essentially 
self-adjoint, it follows that $\Re\D$ and $\Im\D$ are essentially 
self-adjoint on $\Dom\D$. Since $(\til\D_++\til\D_-,\til\D_+-\til\D_-)$ is an almost anti-commuting pair, it follows from \cref{lem:comm_anti-comm} that $(\Re\D,\Im\D)$ is an almost commuting pair.
For all $a\in\A$, we know that $[\til\D_+,a]$ and $[\til\D_-,a]$ are bounded, 
and therefore $[\D,a]$ and $[\D^*,a]$ are also bounded.  
Finally, since the inclusion $(\Dom\D)^{\oplus2}\hookrightarrow E\oplus E$ 
is locally compact, it follows that the inclusion $\Dom\D\hookrightarrow E$ 
is also locally compact. Thus $(\A,E_B,\D)$ is 
indeed an odd indefinite unbounded Kasparov $A$-$B$-module. \qedhere
\end{description}
\end{proof}
We point out that the indefinite Kasparov module $(\A,(E\oplus E)_B,\til\D)$, 
given (as defined above) by the operator
$$
\til\D = \mattwo{0}{\D_+}{\D_-}{0} ,
$$ 
is a very special type of indefinite Kasparov module. 
For instance, its entries $\D_+$ and $\D_-$ are both 
essentially self-adjoint, and they have a common 
core (namely $\Dom\D$). The special nature of such $\til\D$ 
is reflected by the following property of the Wick rotations.

Given a Hilbert $A$-$B$-bimodule ${}_\pi E_B$, recall the \emph{opposite module} ${}_{\pi^\op}E^\op_B$, which is defined as the Hilbert module $E_B$ with the opposite grading (i.e.\ $(E^\op)^0 = E^1$ and $(E^\op)^1 = E^0$), and with the left action $\pi^\op(a) := \pi(a^0) - \pi(a^1)$ for $a = a^0+a^1 \in A^0\oplus A^1 = A$.
\begin{prop}
\label{prop:odd_opp_class}
Let $(\A,E_B,\D)$ be an odd indefinite unbounded Kasparov 
$A$-$B$-module, and consider the corresponding pair of unbounded 
Kasparov modules $(\A,(E\oplus E)_B,\til\D_+,\til\D_-)$ (as in \cref{thm:odd_indef-even_pair}). 
Then $(\A,(E\oplus E)_B,\til\D_+)$ is unitarily equivalent to 
$(\A,(E\oplus E)^\op_B,-\til\D_-)$, and for their $KK$-classes we therefore have 
$[(\A,(E\oplus E)_B,\til\D_+)] = - [(\A,(E\oplus E)_B,\til\D_-)] \in KK(A,B)$.
\end{prop}
\begin{proof}
First, the operators $\til\D_+$ and $-\til\D_-$ are unitarily equivalent:
$$
\mattwo{0}{1}{-1}{0} \mattwo{0}{\D^*}{\D}{0} \mattwo{0}{-1}{1}{0} = \mattwo{0}{-\D}{-\D^*}{0} .
$$
However, we also find that
$$
\mattwo{0}{1}{-1}{0} \mattwo{1}{0}{0}{-1} \mattwo{0}{-1}{1}{0} = \mattwo{-1}{0}{0}{1} ,
$$
so under this unitary equivalence the $\Z_2$-grading becomes the opposite. 
Thus, we have the unitary equivalence 
$(\A,(E\oplus E)_B,\til\D_+) \sim (\A,(E\oplus E)^\op_B,-\til\D_-)$. 
Recalling that the class of $(\A,(E\oplus E)^\op_B,-\til\D_-)$ is the negative of the class of 
$(\A,(E\oplus E)_B,\til\D_-)$, the last statement follows immediately.
\end{proof}
We would like to characterise the types of indefinite 
Kasparov modules that are obtained from odd indefinite 
Kasparov modules, and for this purpose we prove a 
converse to the above proposition.
\begin{prop}
\label{prop:odd_opp_class_converse}
Let $A$ and $B$ be trivially graded $C^*$-algebras. 
Let $(\A,E_B,\D_1,\D_2)$ be a pair of unbounded Kasparov 
$A$-$B$-modules such that $(\A,E_B,\D_1)$ is unitarily equivalent to 
$(\A,E^\op_B,-\D_2)$ via an \emph{anti-self-adjoint} unitary 
$$
\mattwo{0}{-U^*}{U}{0} ,
$$
where $U$ is a unitary isomorphism $E^0\to E^1$ and we identify 
$E^\op \simeq E = E^0\oplus E^1$ as \emph{ungraded} modules. 
Then $(\A,E^0_B,U^* \D_1|_{E^0})$ is an odd indefinite unbounded Kasparov $A$-$B$-module. 
\end{prop}
\begin{remark}
Suppose that $\D_1=\D_2$, so we just have an unbounded Kasparov module $(\A,E_B,\D_1)$. 
The anti-self-adjoint unitary operator given above can be seen as the 
generator of the Clifford algebra $\CCliff_1$. Since it is odd and 
anti-commutes with $\D_1=\D_2$, this means that $(\A,E_B,\D_1)$ extends to
an unbounded Kasparov module $(\A\otimes \CCliff_1,E_B,\D_1)$
and thus  represents 
a class in the \emph{odd} $KK$-theory $KK^1(A,B) = KK(A\otimes\CCliff_1,B)$. 
If $\D_1\neq\D_2$ however, the anti-self-adjoint unitary does \emph{not} 
anti-commute with $\D_1$ 
(nor $\D_2$), so the pair of Kasparov $A$-$B$-modules 
does \emph{not} extend to a pair of Kasparov $A\otimes\CCliff_1$-$B$-modules. 
\end{remark}
\begin{proof}
Using the isomorphism $E^\op \simeq E = E^0\oplus E^1$ as 
\emph{ungraded} modules, any \emph{even} unitary isomorphism 
$E \to E^\op$ can be written in the form
$$
\mattwo{0}{-V^*}{U}{0} ,
$$
where $U$ and $V$ are unitary isomorphisms $E^0\to E^1$. The 
assumption that this unitary isomorphism is anti-self-adjoint implies that $U=V$. If we write the self-adjoint operator $\D_1$ on $E^0\oplus E^1$ as
$$
\D_1 = \mattwo{0}{\D_0^*}{\D_0}{0} ,
$$
the unitary equivalence of $\D_1$ and $-\D_2$ then yields
$$
\D_2 = - \mattwo{0}{-U^*}{U}{0} \mattwo{0}{\D_0^*}{\D_0}{0} \mattwo{0}{U^*}{-U}{0} 
= \mattwo{0}{U^*\D_0U^*}{U\D_0^*U}{0} .
$$
The algebra $A$ is trivially graded, so its representation on 
$E$ and $E^\op$ is the same. Writing $a = a_0\oplus a_1$, we find that 
$a_1 = Ua_0U^*$. Hence the representation of 
$A$ on $E$ is determined by its representation on $E^0$.  
Using the identification $E^1 = U E^0$, we can rewrite $a$, 
$\D_1$, and $\D_2$ as operators on $E^0\oplus E^0$ as
\begin{align*}
a &= \mattwo{a_0}{0}{0}{a_0} , & \D_1 
&= \mattwo{0}{\D_0^*U}{U^*\D_0}{0} , & \D_2 &= \mattwo{0}{U^*\D_0}{\D_0^*U}{0} .
\end{align*}
By defining $\D := U^*\D_0\colon\Dom\D_0\to E^0$, this can be rewritten as
\begin{align*}
\D_1 &= \mattwo{0}{\D^*}{\D}{0} , & \D_2 &= \mattwo{0}{\D}{\D^*}{0} .
\end{align*}
Hence 
it follows from \cref{thm:odd_indef-even_pair} that $(\A,E^0_B,\D)$ 
is an odd indefinite unbounded Kasparov module. 
\end{proof}
We point out that our constructions are well-defined and 
reversible up to unitary equivalence (where we need to 
allow for unitary equivalence because of the 
freedom in the unitary isomorphism $U\colon E^0\to E^1$). 
Combining the previous two propositions with \cref{thm:odd_indef-even_pair}, 
we thus obtain:
\begin{thm}
Let $A$ and $B$ be trivially graded $C^*$-algebras. 
The constructions of \cref{prop:odd_opp_class,prop:odd_opp_class_converse} implement a bijection between 
unitary equivalence classes of \emph{odd} 
indefinite unbounded Kasparov $A$-$B$-modules $(\A,E_B,\D)$ and
unitary equivalence classes of \emph{pairs} of 
unbounded Kasparov $A$-$B$-modules $(\A,E_B,\D_1,\D_2)$ 
such that $(\A,E_B,\D_1)$ is unitarily equivalent to 
$(\A,E^\op_B,-\D_2)$ via an \emph{anti-self-adjoint} even unitary.
\end{thm}
\section{Examples}
\label{sec:egs}
\subsection{Pseudo-Riemannian spin manifolds}
\label{sec:pseudo-Riem-mfd}
In this section we describe the main example for indefinite Kasparov modules, namely pseudo-Riemannian spin manifolds. We briefly recall the construction of the canonical Dirac operator on a pseudo-Riemannian spin manifold, and for more details we refer to \cite{Baum81}. 

Let $(M,g)$ be an $n$-dimensional time- and space-oriented pseudo-Riemannian spin manifold of signature $(t,s)$, where $t$ is the number of time dimensions (for which $g$ is negative-definite) and $s$ is the number of spatial dimensions (for which $g$ is positive-definite). We consider an orthogonal decomposition of the tangent bundle $TM = \bE_t\oplus \bE_s$, which always exists but is far from unique. We will consider elements of $\bE_t$ to be `purely timelike' and elements of $\bE_s$ to be `purely spacelike'. Given our choice of decomposition $TM = \bE_t\oplus \bE_s$, we have a \emph{timelike projection} $T\colon \bE_t\oplus \bE_s \to \bE_t$ and a \emph{spacelike reflection} $r := 1-2T$ which acts as $(-1)\oplus1$ on $\bE_t\oplus \bE_s$. 

Let $\Cliff(TM,g)$ denote the real Clifford algebra with respect to $g$, and denote the Clifford representation $TM \into \Cliff(TM,g)$ by $\gamma$. Our conventions are such that $\gamma(v)\gamma(w) + \gamma(w)\gamma(v) = - 2 g(v,w)$. 
We shall denote by $h$ the map $T^*M\rightarrow TM$ which maps $\alpha\in T^*M$ to its dual in $TM$ with respect to the metric $g$. That is:
\begin{align*}
h(\alpha) &= v   \Longleftrightarrow   \alpha(w) = g(v,w) \quad\text{for all } w\in TM .
\end{align*}
We assume that $M$ is equipped with a spin structure. 
We consider the corresponding spinor bundle $\bS\rightarrow M$ 
and its space of compactly supported, smooth sections 
$\Gamma_c^\infty(\bS)$. We denote by $c$ the pseudo-Riemannian 
Clifford multiplication $\Gamma_c^\infty(T^*M\otimes \bS)\rightarrow \Gamma_c^\infty(\bS)$ 
given by 
\begin{align*}
c(\alpha\otimes\psi) &:= \gamma\big(h(\alpha)\big)\psi .
\end{align*}
Let $\nabla$ be the Levi-Civita connection for the pseudo-Riemannian 
metric $g$, and let $\nabla^\bS$ be its lift to the spinor bundle. 
The Dirac operator on $\Gamma_c^\infty(\bS)$ is defined as the composition
\begin{align*}
\sD &\colon \Gamma_c^\infty(\bS) \xrightarrow{\nabla^\bS} \Gamma_c^\infty(T^*M\otimes \bS) \xrightarrow{c} \Gamma_c^\infty(\bS) . 
\end{align*}
Locally, we can choose a (pseudo-)orthonormal frame $\{e_j\}_{j=1}^n$ 
corresponding to our choice of decomposition $TM = \bE_t\oplus \bE_s$, 
such that $e_j\in \bE_t$ for $j\leq t$ and $e_j\in \bE_s$ for $j>t$. 
In terms of this frame, the metric can be written as
\begin{align*}
g(e_i,e_j) &= \delta_{ij}\kappa(j) , & \kappa(j) &= 
\begin{cases}-1 & j=1,\ldots,t;\\ 1 & j=t+1,\ldots,n. \end{cases}
\end{align*}
Let $\{\theta^i\}_{i=1}^n$ be the basis of $T^*M$ dual to $\{e_j\}_{j=1}^n$, 
so that $\theta^i(e_j) = \delta^i_j$. We then see that $h(\theta^j) = \kappa(j)e_j$. 
In terms of the local frame $\{e_j\}$, we can then write the Dirac operator as
$$
\sD := c \circ \nabla^\bS = \sum_{j=1}^n \kappa(j) \gamma(e_j) \nabla^\bS_{e_j} .
$$
\subsubsection{The Hilbert space of spinors}
Given the decomposition $TM = \bE_t\oplus \bE_s$, there exists a 
positive-definite hermitian structure \cite[\S3.3.1]{Baum81}
$$
(\cdot|\cdot)\colon \Gamma_c^\infty(\bS) \times \Gamma_c^\infty(\bS) \to C_c^\infty(M) ,
$$
which gives rise to the inner product
$
\la\psi_1|\psi_2\ra := \int_M (\psi_1|\psi_2) \dvol_g ,
$
for all $\psi_1,\psi_2\in\Gamma_c^\infty(\bS)$, where $\dvol_g$ 
denotes the canonical volume form of $(M,g)$. The completion of 
$\Gamma_c^\infty(\bS)$ with respect to this inner product is denoted 
$L^2(\bS)$. We can define an operator $\mJ_M$ on $L^2(\bS)$ by setting
$$
\mJ_M := i^{t(t-1)/2} \gamma(e_1) \cdots \gamma(e_t) , 
$$
where $\{e_j\}$ is a local orthonormal frame corresponding to the 
decomposition $TM = \bE_t\oplus \bE_s$. This operator is self-adjoint 
and unitary, and is related to the spacelike reflection $r$ via
$$
\mJ_M\gamma(v)\mJ_M = (-1)^t \gamma(rv) .
$$
The space $L^2(\bS)$ then becomes a Krein space with the indefinite inner product $\la\cdot|\cdot\ra_{\mJ_M} := \la\mJ_M\cdot|\cdot\ra$ and with fundamental symmetry $\mJ_M$. 
\subsubsection{The Dirac operator and its Wick rotations}
Using the spacelike reflection $r$, we can define a `Wick rotated' metric $g_r$ on $M$ by setting
$$
g_r(v,w) := g(rv,w)
$$
for all $v,w\in TM$. One readily checks that $g_r$ is positive-definite, and hence $(M,g_r)$ is a Riemannian manifold. 
Throughout the remainder of this section we make the following assumption:
\begin{assumption}
\label{ass:mfd_complete}
Let $(M,g)$ be an $n$-dimensional time- and space-oriented pseudo-Riemannian spin manifold of signature $(t,s)$, and let $r$ be a spacelike reflection, such that the associated Riemannian metric $g_r$ is complete. 
\end{assumption}
Consider the Dirac operator $\sD := c\circ\nabla^\bS$ on the Hilbert space $L^2(\bS)$ with initial domain $\Gamma_c^\infty(\bS)$. 
From \cite[Satz 3.17]{Baum81} we know that the adjoint of the Dirac operator $\sD = \sum_{j=1}^{n} \kappa(j) \gamma(e_j) \nabla^\bS_{e_j}$ takes the form 
$$
\sD^* = \sum_{j=1}^n \gamma(e_j) \mJ_M \nabla^\bS_{e_j} \mJ_M = \sum_{j=1}^n \left( \gamma(e_j) \nabla^\bS_{e_j} + \gamma(e_j) \mJ_M \big[\nabla^\bS_{e_j},\mJ_M\big] \right) .
$$
For the real and imaginary parts of $\sD$ we thus find
\begin{align*}
\Re\sD &= \sum_{j=t+1}^n \gamma(e_j) \nabla^\bS_{e_j} + \frac12 \sum_{j=1}^n \gamma(e_j) \mJ_M \big[\nabla^\bS_{e_j},\mJ_M\big] , & \Im\sD &= i \sum_{j=1}^t \gamma(e_j) \nabla^\bS_{e_j} + \frac i2 \sum_{j=1}^n \gamma(e_j) \mJ_M \big[\nabla^\bS_{e_j},\mJ_M\big].
\end{align*}
The assumption that $g_r$ is complete implies that $\Re\sD$ and $\Im\sD$ are essentially self-adjoint \cite[Satz 3.19]{Baum81}. 
For the Wick rotations we then obtain the formula
\begin{align*}
\sD_\pm &= \pm i \sum_{j=1}^t \gamma(e_j) \nabla^\bS_{e_j} + \sum_{j=t+1}^n \gamma(e_j) \nabla^\bS_{e_j} + \frac{1\pm i}2 \sum_{j=1}^n \gamma(e_j) \mJ_M \big[\nabla^\bS_{e_j},\mJ_M\big] \\
&= \sum_{j=1}^n \gamma_\pm(e_j) \nabla^\bS_{e_j} + \frac{1\pm i}2 \sum_{j=1}^n \gamma(e_j) \mJ_M \big[\nabla^\bS_{e_j},\mJ_M\big],
\end{align*}
where we have defined the `Wick rotated' Clifford representations $\gamma_\pm$ as
\begin{align}
\label{eq:Wick_Cliff}
\gamma_\pm(v) &:= \pm i \gamma(v_t) + \gamma(v_s)
\end{align}
for any $v = v_t+v_s \in \bE_t\oplus \bE_s=TM$. Since $\gamma_\pm(v)^2 = -\gamma(v_t)^2 + \gamma(v_s)^2 = g(v_t,v_t) - g(v_s,v_s) = - g_r(v,v)$, we see that $\gamma_\pm$ (for either choice of sign) is a Clifford representation associated to the Riemannian metric $g_r$. 

\begin{prop}
\label{prop:sd_pm_pair}
Let $(M,g,r)$ be as in \cref{ass:mfd_complete}.
Then the Wick rotations $\sD_\pm$  yield spectral triples 
$(C_c^\infty(M), L^2(\bS), \sD_\pm)$ such that $\Dom\sD_+ \cap \Dom\sD_-$ 
is a common core for $\D_+$ and $\D_-$, and such that $\sD_+\pm\sD_-$ 
is essentially self-adjoint on this domain. 
\end{prop}
\begin{proof}
The essential self-adjointness of $\sD_\pm$ on $\Gamma_c^\infty(\bS)$ 
follows from the completeness of $g_r$ (see e.g.\ \cite[Proposition 10.2.11]{Higson-Roe00}). 
Commutators of $\sD_\pm$ with functions $f\in C_c^\infty(M)$ are 
bounded because $\sD_\pm$ is a first-order differential operator, 
whose coefficients are smooth and hence bounded on any compact set.
For a vector $v\in TM$, the principal symbol of $\sD_\pm$ is given by 
$i\gamma_\pm(v)$. Since the square of the principal symbol equals 
the positive-definite metric $g_r(v,v)$, this implies that $\sD_\pm$ is elliptic, 
and hence it has locally compact resolvent 
(see e.g.\ \cite[Proposition 10.5.2]{Higson-Roe00}). 
Thus we indeed have spectral triples $(C_c^\infty(M), L^2(\bS), \sD_\pm)$. 

For the domains of the Wick rotations we have 
$\Dom\sD_+ \cap \Dom\sD_- \supset \Dom\sD\cap\Dom\sD^*$. 
Since this domain contains $\Gamma_c^\infty(\bS)$, it is a core for both $\sD_+$ and $\sD_-$. 
The operators $\Re\sD$ and $\Im\sD$ are essentially self-adjoint on $\Gamma_c^\infty(\bS)$ by 
\cite[Satz 3.19]{Baum81}, 
and since they can be extended to symmetric operators on 
$\Dom\sD\cap\Dom\sD^*$, it follows that these symmetric 
extensions are also essentially self-adjoint. Thus $\Re\sD = \frac12(\sD_++\sD_-)$ 
and $\Im\sD = \frac12(\sD_+-\sD_-)$ are essentially self-adjoint on $\Dom\sD_+\cap\Dom\sD_-$. 
\end{proof}

\begin{remark}
The above proposition shows that (under only mild assumptions) 
a pseudo-Riemannian spin manifold gives rise to two spectral triples 
satisfying the first and second conditions in \cref{defn:even_pair}. From the reverse Wick rotation of \cref{prop:pair_to_indef}, we then \emph{almost} obtain an indefinite spectral triple $(C_c^\infty(M), L^2(\bS), \sD)$, except that $\Re\sD$ and $\Im\sD$ do \emph{not} almost anti-commute. 
Indeed, although the anti-commutator $\{\Re\sD,\Im\sD\}$ is a first-order differential operator, it contains in general both spacelike derivatives and timelike derivatives, and thus it is not relatively bounded by $\Re\sD$. In order to ensure that $\Re\sD$ and $\Im\sD$ almost anti-commute, we need the timelike part of $\{\Re\sD,\Im\sD\}$ to vanish identically. In the next 
subsection, we will provide sufficient conditions on a Lorentzian manifold 
to ensure that $\Re\sD$ and $\Im\sD$ almost anti-commute.  

As mentioned in the introduction, we emphasise however that the main reason for imposing this almost anti-commuting condition is to prove self-adjointness of the Wick rotations $\D_\pm$. For the Dirac operator $\sD$, we can simply prove the self-adjointness of $\sD_\pm$ directly (as we did in \cref{prop:sd_pm_pair}), and this condition is therefore not necessary for describing pseudo-Riemannian manifolds.
\end{remark}
\subsubsection{Lorentzian manifolds with parallel time}
\begin{defn}
\label{defn:bdd_geom}
Let $(M,g,r)$ be as in \cref{ass:mfd_complete}. We say that $(M,g,r)$ 
has bounded geometry if $(M,g_r)$ has strictly positive injectivity radius, 
and all the covariant derivatives of the (pseudo-Riemannian) 
curvature tensor of $(M,g)$ are bounded (with respect to $g_r$) on $M$.
A Dirac bundle on $M$ is said to have bounded geometry if in 
addition all the  covariant derivatives  
of $\Omega^\bS$, the curvature tensor of the connection 
$\nabla^\bS$, are bounded (w.r.t $g_r$) on $M$. 
For brevity, we simply say that $(M,g,r,\bS)$ has bounded geometry.
\end{defn}
We will now restrict to Lorentzian signature, and impose additional assumptions on the geometry:
\begin{assumption}
\label{assum:Lor_parallel}
Let $(M,g)$ be an even-dimensional time- and space-oriented 
Lorentzian spin manifold of signature $(1,n-1)$, with a given spinor bundle $\bS\to M$. 
Let $r$ be a spacelike reflection, such that the associated 
Riemannian metric $g_r$ is complete. Assume furthermore that 
$(M,g,r,\bS)$ has bounded geometry. Lastly, we assume that the 
spacelike reflection $r$ is \emph{parallel} (i.e.\ the unit timelike vector field 
$e_0\in\Gamma(\bE_t)$, corresponding to the decomposition 
$TM = \bE_t\oplus\bE_s$, is parallel: $\nabla e_0=0$). 
\end{assumption}
We choose a local orthonormal frame $\{e_k\}_{k=0}^{n-1}$ 
corresponding to the decomposition $TM = \bE_t\oplus\bE_s$ 
(i.e.\ $e_0$ is timelike and $e_k$ is spacelike for $k>0$). 
The assumption that $e_0$ is parallel then implies that 
$[\nabla^\bS,\gamma(e_0)]=0$. The expressions for the 
real and imaginary parts of $\sD$ and its Wick rotations then simplify to:
\begin{align*}
\Re\sD &= \sum_{j=1}^{n-1} \gamma(e_j) \nabla^\bS_{e_j} , & 
\Im\sD &=  i \gamma(e_0) \nabla^\bS_{e_0} , & 
\sD_\pm &= \sum_{k=0}^{n-1} \gamma_\pm(e_k) \nabla^\bS_{e_k} ,
\end{align*}
where we recall from \cref{eq:Wick_Cliff} the Wick rotated Clifford representations $\gamma_\pm(v) :=  \pm i \gamma(v_t) + \gamma(v_s)$ (for $v = v_t+v_s \in E_t\oplus E_s=TM$). 
\begin{lem}
\label{lem:Re-Im_sD}
Let $(M,g,r,\bS)$ be as in Assumption \ref{assum:Lor_parallel}.
The operators $\Re\sD$ and $\Im\sD$ yield an almost anti-commuting pair $(\Re\sD,\Im\sD)$. 
\end{lem}
\begin{proof}
We observe that the space $\E := \Gamma_c^\infty(\bS)$ of smooth compactly supported sections satisfies conditions 1) and 2) in \cref{defn:almost_(anti-)commute}.
Since $\gamma(e_0)$ commutes with $\nabla^\bS$ and anti-commutes with $\gamma(e_j)$ (for $j\neq0$), we calculate (on $\E$)
\begin{align}
\{\Re\sD,\Im\sD\} &= i \sum_{j=1}^{n-1} \left( \gamma(e_j) \nabla^\bS_{e_j}  \gamma(e_0) \nabla^\bS_{e_0} + \gamma(e_0) \nabla^\bS_{e_0} \gamma(e_j) \nabla^\bS_{e_j} \right) \notag\\
&= i \sum_{j=1}^{n-1} \Big( \gamma(e_j) \gamma(e_0) \nabla^\bS_{e_j} \nabla^\bS_{e_0} + \gamma(e_0) \gamma(e_j) \nabla^\bS_{e_0} \nabla^\bS_{e_j} + \gamma(e_0) \big[\nabla^\bS_{e_0},\gamma(e_j)\big] \nabla^\bS_{e_j} \Big) \notag\\
&= i \sum_{j=1}^{n-1} \left( \gamma(e_0) \gamma(e_j) \big[\nabla^\bS_{e_0},\nabla^\bS_{e_j}\big] + \gamma(e_0) \big[\nabla^\bS_{e_0},\gamma(e_j)\big] \nabla^\bS_{e_j} \right) \notag\\
&= i \sum_{j=1}^{n-1} \Big( \gamma(e_0) \gamma(e_j) \nabla^\bS_{[e_0,e_j]} + \gamma(e_0) \gamma(e_j) \Omega^\bS(e_0,e_j) + \gamma(e_0) \big[\nabla^\bS_{e_0},\gamma(e_j)\big] \nabla^\bS_{e_j} \Big) . \label{eq:comm_Re_Im_sD}
\end{align}
The curvature $\Omega^\bS(e_0,e_j)$ and the commutator $[\nabla^\bS_{e_0},\gamma(e_j)]$ are bounded by the assumption of bounded geometry.
Hence, on the last line of \cref{eq:comm_Re_Im_sD}, the second term is bounded and the third term is relatively bounded by $\Re\sD$. Since $e_0$ is parallel, we have $[e_0,e_j] = \nabla_{e_0}e_j - \nabla_{e_j}e_0 = \nabla_{e_0}e_j$. Since $e_0$ and $e_j$ are orthogonal (for $j>0$), we find that
$$
g(\nabla_{e_0}e_j,e_0) = - g(e_j,\nabla_{e_0}e_0) + e_0\big(g(e_j,e_0)\big) = 0 ,
$$
and hence $[e_0,e_j] = \nabla_{e_0}e_j \in \bE_s$. This means that the first term in \cref{eq:comm_Re_Im_sD} also only has spacelike derivatives, and is therefore relatively bounded by $\Re\sD$ as well. 
\end{proof}
Combining this with \cref{prop:sd_pm_pair}, and applying the reverse Wick rotation of \cref{prop:pair_to_indef}, we obtain:
\begin{coro}
Let $(M,g,r,\bS)$ satisfy \cref{assum:Lor_parallel}.
The Wick rotations $\sD_\pm$ yield an even pair of spectral triples $(C_c^\infty(M),L^2(\bS),\sD_+,\sD_-)$, and hence the triple $(C_c^\infty(M),L^2(\bS),\sD)$ is an even indefinite spectral triple.
\end{coro}
Finally, we relate the Wick rotations $\sD_\pm$ of the Lorentzian Dirac operator on $(M,g)$ to the canonical Dirac operator on the Riemannian manifold $(M,g_r)$. Since $M$ is even-dimensional, recall that the spinor bundle $\bS$ is $\Z_2$-graded with the grading operator given by
\begin{align}
\label{eq:spin_grading}
\Gamma_M := i^{-t+\frac{n(n+1)}{2}} \gamma(e_0) \cdots \gamma(e_{n-1}) ,
\end{align}
where for the Lorentzian signature we of course have $t=1$. 
\begin{prop}
\label{prop:Wick_commuting_diagram}
The Wick rotations $\sD_\pm$ of the Lorentzian Dirac operator $\sD$ on $(M,g,r,\bS)$ are the two canonical Dirac operators on the Wick rotated Riemannian spin manifold $(M,g_r,\bS)$ corresponding to the two possible choices of orientation $\Gamma_M^\pm$ on $\bS$. In other words, the following diagram commutes.
\begin{align*}
\begin{array}{ccc}
(M,g,r,\bS,\Gamma_M) & \xrightarrow{\textnormal{Wick rotate}} & (M,g_r,\bS,\Gamma_M^\pm) \\
\Big\downarrow && \Big\downarrow \\
\sD & \xrightarrow{\textnormal{Wick rotate}} & \sD_\pm \\
\end{array}
\end{align*}
\end{prop}
\begin{proof}
In \cref{eq:Wick_Cliff} we have given two Clifford representations $\gamma_\pm$ corresponding to the Riemannian metric $g_r$. The grading operators corresponding to these Clifford representations  are given by
$$
\Gamma_M^\pm := i^{n(n+1)/2} \gamma_\pm(e_0) \cdots \gamma_\pm(e_{n-1}) 
= \pm i^{1+n(n+1)/2} \gamma(e_0) \cdots \gamma(e_{n-1}) 
= \mp \Gamma_M ,
$$
where $\Gamma_M$ is given by \cref{eq:spin_grading}. Hence the choice of sign for the Wick rotation of $\gamma$ corresponds to the choice of orientation for the spinor bundle $\bS$ (in the terminology of \cite[\S2.7]{Ply86}, the choice $(\bS,\Gamma_M^-)$ is the reverse spin structure of $(\bS,\Gamma_M^+)$). Next, the assumption that the spacelike reflection $r$ is parallel implies that the Levi-Civita connection $\nabla$ of $g$ is also the Levi-Civita connection for the Riemannian metric $g_r$. Hence the canonical Dirac operators corresponding to each of the orientations are given by $\sD_\pm = \sum_{j=0}^{n-1} \gamma_\pm(e_j) \nabla^\bS_{e_j}$, which are precisely the Wick rotations of the Lorentzian Dirac operator $\sD$. 
\end{proof}
The above proposition motivates our use of the term `Wick rotations' for $\sD_\pm$, as they are precisely the Dirac operators corresponding to the `Wick rotated' metric $g_r$. 
\subsection{The harmonic oscillator}
\label{sec:harm-osc}
The $d$-dimensional harmonic oscillator has been discussed in \cite[\S2.1]{GaW13} (see also \cite{Wul10}), 
where the harmonic oscillator is `deformed' to obtain a description of the spectral geometry of the (noncommutative) Moyal plane with harmonic propagation. 
Here, we only consider the classical (commutative) case. 

On
$L^2(\R^d)$ we consider the bosonic annihilation and creation operators with canonical commutation relations:
\begin{align*}
a_\mu &:= \omega x_\mu + \partial_\mu , & a_\mu^* &= \omega x_\mu - \partial_\mu , & [a_\mu,a_\nu]=[a_\mu^*,a_\nu^*] &= 0 , & [a_\mu,a_\nu^*] &= 2 \omega \delta_{\mu\nu} ,
\end{align*}
for $\mu,\nu=1,\ldots,d$. Here we have also introduced a frequency parameter $\omega>0$. On the exterior algebra $\Lambda(\C^d)$, we introduce the fermionic partners $b_\mu,b_\mu^*$ satisfying the anti-commutation relations
\begin{align*}
\{b_\mu,b_\nu\} = \{b_\mu^*,b_\nu^*\} &= 0 , & \{b_\mu,b_\nu^*\} &= \delta_{\mu\nu} .
\end{align*}

Denote by $|0\ra_f$ the fermionic vacuum satisfying $b_\mu|0\ra_f = 0$ for all $\mu$. By repeated application of the creation operators $b_\mu^*$ one constructs out of this vacuum the $2^d$-dimensional Hilbert space $\Lambda(\C^d) \simeq \C^{2^d}$, yielding the standard orthonormal basis elements $(b_1^*)^{s_1}\cdots(b_d^*)^{s_d}|0\ra_f$ (where $s_\mu\in\{0,1\}$). The fermionic number operator $N_f := \sum_{\mu=1}^d b_\mu^* b_\mu$ naturally defines an $\N$-grading $\Lambda(\C^d) = \bigoplus_{p=0}^d \Lambda^p(\C^d)$ such that $b_\mu\colon\Lambda^p(\C^d)\to\Lambda^{p-1}(\C^d)$ and $b_\mu^*\colon\Lambda^p(\C^d)\to\Lambda^{p+1}(\C^d)$. The induced $\Z_2$-grading $\Gamma$ on $\Lambda(\C^d)$ then satisfies
\begin{align*}
\Gamma &= (-1)^{N_f} , & \Gamma^2 &= 1, & \Gamma^* &= \Gamma , & \Gamma b_\mu &= - b_\mu \Gamma , & \Gamma b_\mu^* &= - b_\mu^* \Gamma .
\end{align*}
Thus we obtain a $\Z_2$-grading on the Hilbert space $L^2(\R^d)\otimes\Lambda(\C^d)$ given by $1\otimes\Gamma$, which we will also simply denote by $\Gamma$. 
On this $\Z_2$-graded Hilbert space $L^2(\R^d) \otimes \Lambda(\C^d)$ we then consider the odd operators
\begin{align*}
\D_1 &:= \sum_{\mu=1}^d \left( a_\mu \otimes b_\mu^* + a_\mu^* \otimes b_\mu \right) , & \D_2 &:= \sum_{\mu=1}^d \left( a_\mu \otimes b_\mu + a_\mu^* \otimes b_\mu^* \right) .
\end{align*}
Their squares are of the form
\begin{align*}
\D_1^2 &= H \otimes 1 + \omega \otimes \Sigma , & \D_2^2 &= H \otimes 1 - \omega \otimes \Sigma ,
\end{align*}
where the \emph{Hamiltonian} $H$ and the \emph{spin matrix} $\Sigma$ are defined as
\begin{align*}
H &:= \sum_{\mu=1}^d \left(\omega^2 x_\mu^2 - \partial_\mu^2\right) , & \Sigma &:= \sum_{\mu=1}^d [b_\mu^*,b_\mu] .
\end{align*}
\begin{remark}
Note that in \cite{GaW13} the operator $\D_2$ is defined instead as 
$\D_2 := \sum_{\mu=1}^d \left( i a_\mu \otimes b_\mu - i a_\mu^* \otimes b_\mu^* \right)$. 
However, our definition and theirs yield the same square
$$
\D_2^2 = \left(\sum_{\mu=1}^d \left( a_\mu \otimes b_\mu + a_\mu^* \otimes b_\mu^* \right)\right)^2 = \left(\sum_{\mu=1}^d \left( i a_\mu \otimes b_\mu - i a_\mu^* \otimes b_\mu^* \right) \right)^2 = \left\{ \sum_{\mu=1}^d a_\mu \otimes b_\mu , \sum_{\mu=1}^d a_\mu^* \otimes b_\mu^* \right\} .
$$
\end{remark}
\begin{prop} The data
$(\mS(\R^d), L^2(\R^d)\otimes\Lambda(\C^d), \D_1, \D_2)$ defines an even pair of spectral triples. 
\end{prop}
\begin{proof}
The operator $H$ is well-known to be essentially self-adjoint on 
$\mS(\R^d)$ and to have compact resolvent. Since 
$\omega\otimes\Sigma$ is only a bounded perturbation of 
$H\otimes1$, it follows that $\D_1^2$ and $\D_2^2$ are 
essentially self-adjoint on $\mS(\R^d)\otimes\Lambda(\C^d)$ 
and also have compact resolvent. Since $\D_1$ and $\D_2$ are symmetric and their squares are essentially self-adjoint, it follows (see e.g.\ \cite[exercise 28, Chapter X]{Reed-Simon-II} or the proof of \cite[Lemma 3]{Ber68}) that $\D_1$ and $\D_2$ are also essentially self-adjoint. Likewise, compactness of their resolvents follows from the compactness of the resolvents of their squares. 
Furthermore, commutators of 
$\D_1$ and $\D_2$ with Schwartz functions are bounded. Hence 
$\D_1$ and $\D_2$ indeed yield even spectral triples. 

To show that these spectral triples in fact form an even pair, 
we need to check the axioms in \cref{defn:even_pair}. Since 
$\D_1^2-\D_2^2$ is bounded, it follows that $\Dom\D_1 = \Dom\D_2$.
Furthermore, the operators
\begin{align*}
\D_1+\D_2 &= \sum_{\mu=1}^d (a_\mu+a_\mu^*)\otimes(b_\mu+b_\mu^*) = \sum_{\mu=1}^d 2\omega x_\mu \otimes (b_\mu+b_\mu^*) , \\ 
\D_1-\D_2 &= \sum_{\mu=1}^d (a_\mu^*-a_\mu)\otimes(b_\mu-b_\mu^*) = \sum_{\mu=1}^d -2\partial_\mu \otimes (b_\mu-b_\mu^*) ,
\end{align*}
are essentially self-adjoint on $\mS(\R^d)\otimes\Lambda(\C^d)\subset\Dom\D_1=\Dom\D_2$. Since the graph norm of $\D_1\pm\D_2$ is bounded by the norm $\|\cdot\|_{\D_1,\D_2}$ (cf.\ \cref{lem:norms-1-2}), it follows that the domain of the closure of $\D_1\pm\D_2$ contains $\Dom\D_1\cap\Dom\D_2$, so that $\D_1\pm\D_2$ is also essentially self-adjoint on $\Dom\D_1=\Dom\D_2$. Lastly, the domain $\E := \mS(\R^d)\otimes\Lambda(\C^d)$ satisfies properties 1) and 2) in \cref{defn:almost_(anti-)commute}, and the operator $\{\D_1+\D_2,\D_1-\D_2\} = \D_1^2-\D_2^2$ is bounded on this domain, so $(\D_1+\D_2,\D_1-\D_2)$ is an almost anti-commuting pair. 
\end{proof}
From \cref{prop:pair_to_indef} we then obtain:
\begin{coro}
The operator 
$$
\D := \frac12(\D_1+\D_2) + \frac i2 (\D_1-\D_2) = \sum_{\mu=1}^d \left( \omega x_\mu \otimes (b_\mu+b_\mu^*) - i \partial_\mu \otimes (b_\mu-b_\mu^*) \right) 
$$
yields an even indefinite spectral triple $(\mS(\R^d), L^2(\R^d)\otimes\Lambda(\C^d), \D)$.
\end{coro}
We remark that this operator $\D$ still encodes all the information of the $d$-dimensional harmonic oscillator. In particular, the Hamiltonian $H$ and the spin matrix $\Sigma$ can be recovered via
\begin{align*}
\frac12(\D\D^*+\D^*\D) &= \frac12(\D_1^2+\D_2^2) = H\otimes1 , & -\frac i2(\D^2-{\D^*}^2) &= \frac12(\D_1^2-\D_2^2) = \omega\otimes\Sigma .
\end{align*}
\begin{example}
Suppose that $d=1$. We then have the operators
\begin{align*}
a &:= \omega x + \frac{d}{dx} , & a^* &= \omega x - \frac{d}{dx} , & b &:= \mattwo{0}{1}{0}{0} , & b^* &= \mattwo{0}{0}{1}{0} ,
\end{align*}
acting on the Hilbert space $L^2(\R) \otimes \C^2$.
These operators give rise to two self-adjoint operators $\D_1$ and $\D_2$ and their reverse Wick rotation $\D$ given by
\begin{align*}
\D_1 &:= a \otimes b^* + a^* \otimes b = \mattwo{0}{a^*}{a}{0} , & \D_2 &:= a \otimes b + a^* \otimes b^* = \mattwo{0}{a}{a^*}{0} , & \D &= \mattwo{0}{\omega x-i\frac{d}{dx}}{\omega x+i\frac{d}{dx}}{0} .
\end{align*}
From \cref{thm:odd_indef-even_pair} we then see that the $1$-dimensional harmonic oscillator yields an \emph{odd} indefinite spectral triple $(\mS(\R), L^2(\R), \omega x+\frac{d}{dx})$. 
\end{example}
\subsection{Families of spectral triples}
\label{sec:families}
In this section we study families of spectral triples 
$\{(\A,{}_{\pi_x}\mH,\D_1(x))\}_{x\in M}$ parametrised by a 
Riemannian manifold $M$. We use these families to 
construct examples of pairs of spectral triples and thus of 
indefinite spectral triples. Our approach is largely based on and 
inspired by work of Kaad and Lesch \cite[\S8]{KL13}, who 
studied the spectral flow of a family of operators $\{\D_1(x)\}_{x\in M}$.
\subsubsection{The family of spectral triples}
Let us start with a brief discussion of families of operators parametrised by the manifold $M$.
\begin{defn}
A map $S(\cdot)\colon M\to\mL(\mH_1,\mH_2)$, $x\mapsto S(x)$, is said to have a \emph{uniformly bounded weak derivative} if the map is weakly differentiable (i.e.\ the map $x\mapsto \la S(x)\xi,\eta\ra$ is differentiable for each $\xi\in\mH_1$ and $\eta\in\mH_2$), the weak derivative $dS(x)\colon \mH_1 \to \mH_2\otimes T_x^*(M)$ is bounded for all $x\in M$, and the supremum $\sup_{x\in M}\|dS(x)\|$ is finite.
\end{defn}
We gather a few statements from \cite[\S8]{KL13} into the following lemma. 
\begin{lem}
\label{lem:fam_op}
Let $S(\cdot)\colon M\to\mL(\mH_1,\mH_2)$ have a uniformly bounded weak derivative. 
Then: \\
1) If $x,y\in M$ lie in the same coordinate chart, then 
$$
\big\| S(x) - S(y) \big\| \leq \sup_{z\in M} \|dS(z)\| \cdot \dist(x,y) ,
$$
where $\dist(x,y)$ denotes the geodesic distance between $x$ and $y$;\\
 2) If $\sup_{x\in M} \|S(x)\|\leq\infty$, then $S(\cdot)$ yields a well-defined operator $C_0(M,\mH_1)\to C_0(M,\mH_2)$ by setting
$$
\big(S(\cdot)\psi\big)(x) := S(x) \psi(x) ,
$$ 
for $\psi\in C_0(M,\mH_1)$.
\end{lem}
\begin{proof}
We refer to \cite[Remark 8.4, 2.]{KL13} for a short proof of 1). 
For 2) we need to check that $S(x)\psi(x)$ is continuous in $x$. We have the inequality
$$
\big\| S(x) \psi(x) - S(y) \psi(y) \big\| 
\leq \big\| S(x) - S(y) \big\| \big\| \psi(x) \big\| + \big\| S(y) \big\| \big\| \psi(x) - \psi(y) \big\| .
$$
As $y\to x$, each of these terms approaches zero; 
the first term by the first statement of this lemma, the second by continuity of $\psi$. 
\end{proof}
\begin{defn}
\label{defn:family}
A \emph{weakly differentiable family of spectral triples} $\{(\A,{}_{\pi_x}\mH,\D_1(x))\}_{x\in M}$ parametrised by the manifold $M$ is a family of spectral triples $\{(\A,{}_{\pi_x}\mH,\D_1(x))\}_{x\in M}$ such that the following conditions are satisfied:\vspace{-6pt}
\begin{itemize}
\item there exists another Hilbert space $W$ which is continuously and densely embedded in $\mH$ such that the inclusion map $\iota\colon W\hookrightarrow\mH$ is \emph{locally compact}, i.e.\ the composition $\pi_x(a)\circ\iota$ is compact for each $x\in M$ and $a\in\A$; \vspace{-6pt}
\item the domain of $\D_1(x)$ is independent of $x$ and equals $W$, and the graph norm of $\D_1(x)$ is uniformly equivalent to the norm of $W$ (i.e.\ there exist constants $C_1,C_2>0$ such that $C_1\|\xi\|_W \leq \|\xi\|_{\D_1(x)} \leq C_2\|\xi\|_W$ for all $\xi\in W$ and all $x\in M$);\vspace{-6pt}
\item for each $a\in\A$, the maps $\D_1(\cdot)\colon M\to\mL(W,\mH)$ and $\pi_\cdot(a)\colon M\to\mL(\mH)$ have uniformly bounded weak derivatives, and the map $[\D_1(\cdot),\pi_\cdot(a)]\colon M\to\mL(\mH)$ is continuous. 
\end{itemize}
\end{defn}
\begin{remark}
\begin{enumerate}
\item The unbounded operator $\D_1(x)\colon\Dom\D_1(x)\to\mH$ is considered a bounded operator $W\to\mH$, where $W=\Dom\D_1(x)$ is a Hilbert space with respect to the graph inner product of $\D_1(x)$. Since the graph norms of $\D_1(x)$ are equivalent for all $x\in M$, it follows that the bound on the operator norm of $d\D_1(x)$ is thus a relative bound with respect to $\D_1(y)$, for any $y\in M$. \vspace{-6pt}
\item The requirement that the graph norm of $\D_1(x)$ is \emph{uniformly} equivalent to the norm of $W$ implies that $\sup_{x\in M} \|\D_1(x)\|$ is finite. \vspace{-6pt}
\item The case where $\A=\C$ and $\pi_x$ is scalar multiplication brings us back to the case of a family of operators $\{\D_1(x)\}$ as studied in \cite[\S8]{KL13}. 
\end{enumerate}
\end{remark}
Consider the Hilbert $C_0(M)$-module $C_0(M,\mH)$. The family of representations $\pi_x\colon A\to\mL(\mH)$ determines a representation $\pi\colon A\otimes C_0(M) \simeq C_0(M,A)\to C_0(M,\mL(\mH)) \simeq \End_{C_0(M)}(C_0(M,\mH))$ by setting
$$
(\pi(a)\psi)(x) := \pi_x(a(x)) \psi(x) ,
$$
for $\psi\in C_0(M,\mH)$ and $a\in C_0(M,A)$. 
The family of operators $\{\D_1(x)\}$ on the Hilbert space 
$\mH$ defines a new operator $\D_1(\cdot)$ on the 
$C_0(M)$-module $C_0(M,\mH)$ with domain $C_0(M,W)$ by setting
$$
(\D_1(\cdot)\psi)(x) := \D_1(x)\psi(x) .
$$
The assumption of weak differentiability is more than 
sufficient to ensure that $\pi$ and $\D_1(\cdot)$ are well-defined (see \cref{lem:fam_op}, part 2). 
The operator $\D_1(\cdot)\colon C_0(M,W) \to C_0(M,\mH)$ is densely defined and symmetric. 
\begin{remark}
In \cite[\S8]{KL13} the family $\{\D_1(x)\}_{x\in M}$ is used to construct 
a class in the odd K-theory $K_1(C_0(M)) = KK^1(\C,C_0(M))$ of 
$C_0(M)$. In order to ensure that $\D_1(\cdot)$ has compact resolvent 
(as an operator on the right $C_0(M)$-module $C_0(M,\mH)$), 
it is then necessary to replace $\D_1(\cdot)$ by $f^{-1}\D_1(\cdot)$, for a 
strictly positive function $f\in C_0^1(M)$. In our approach we aim to construct instead 
a class in $KK^1(C_0(M,A),C_0(M))$, for which introducing this 
function $f$ is not necessary, as now we only need the 
resolvent to be \emph{locally} compact (for the left action by $C_0(M,A)$). \end{remark}
\begin{prop}[{cf.\ \cite[Proposition 8.7]{KL13}}]
\label{prop:fam_Kasp_mod}
If $\{(\A,{}_{\pi_x}\mH,\D_1(x))\}_{x\in M}$ is a weakly 
differentiable family of spectral triples, then the triple $(\A\odot C_c^\infty(M), C_0(M,\mH)_{C_0(M)},\D_1(\cdot))$ 
is an odd unbounded Kasparov $C_0(M,A)$-$C_0(M)$-module. 
\end{prop}
\begin{proof}
The operator $\D_1(x)$ is self-adjoint for each $x\in M$. It then follows from the 
local-global principle \cite[Theorems 4.2, 5.6, and 5.8]{KL12} that $\D_1(\cdot)$ is self-adjoint and regular. However, this can also be seen directly. First, observe that the resolvent $(\D_1(x)\pm i)^{-1}$ depends continuously 
on $x$, since by the resolvent identity and the first statement of \cref{lem:fam_op} we have
\begin{align*}
\big\|(\D_1(x)\pm i)^{-1} - (\D_1(y)\pm i)^{-1}\big\| 
&= \big\|(\D_1(x)\pm i)^{-1} (\D_1(y) - \D_1(x)) (\D_1(y)\pm i)^{-1} \big\| \\
&\leq \big\|(\D_1(x)\pm i)^{-1}\big\| \; \big\|\D_1(y) - \D_1(x)\big\| \; \big\|(\D_1(y)\pm i)^{-1} \big\| \\
&\leq \big\|\D_1(y) - \D_1(x)\big\| 
\leq \sup_{z\in M} \|d(\D_1(z))\| \cdot \dist(x,y) .
\end{align*}
Since $\D_1(x)\pm i$ is surjective for each $x\in M$, this implies that $\D_1(\cdot)\pm i$ is also surjective, and hence $\D_1(\cdot)$ is self-adjoint and regular. 

The algebraic tensor product $\A\odot C_c^\infty(M)$ is dense in $C_0(M,A)$, and for 
$a\otimes f\in \A\odot C_c^\infty(M)$ the commutators 
$$
\big[ \D_1(\cdot) , \pi(a\otimes f) \big](x) = f(x) \big[\D_1(x),\pi_x(a)\big]  
$$
are bounded for each $x$. By assumption such commutators are 
continuous, and the compact support of $f$ then ensures that they are globally bounded. 

It remains to show that $\pi(a\otimes f) (\D_1(\cdot)\pm i)^{-1}$ is 
compact (as an operator on the $C_0(M)$-module 
$C_0(M,\mH)$) for each $a\in A$ and $f\in C_0(M)$. 
The 
compact operators on $C_0(M,\mH)$ are given by $C_0(M,\mK(\mH))$. 
The operator $\pi_x(a) (\D_1(x)\pm i)^{-1}$ is compact and bounded by 
$\|a\|$ for each $x\in M$ (since $(\A,{}_{\pi_x}\mH,\D_1(x))$ is a spectral triple). 
Hence the map $M\to \mK(\mH)$, $x \mapsto \pi_x(a)(\D_1(x)\pm i)^{-1}$ 
is continuous and globally bounded by $\|a\|$, so if we also multiply by 
$f\in C_0(M)$ we get $\pi(a\otimes f)(\D_1(\cdot)\pm i)^{-1} \in C_0(M,\mK(\mH))$.
\end{proof}
\subsubsection{The Kasparov product}
We would like to use the Kasparov product to `glue together' our family of spectral triples. For this purpose, we need to consider a spectral triple on the manifold $M$, which we construct as follows. 
From here on we will assume that the Riemannian manifold $M$ is complete. 
Consider a first-order symmetric elliptic differential operator $\D_2\colon \Gamma_c^\infty(M,F)\to\Gamma_c^\infty(M,F)$ 
on a hermitian 
vector bundle $F\to M$, which has \emph{bounded propagation speed}, 
i.e.\ the principal symbol $\sigma_{\D_2}\colon T^*M\to\End(F)$ satisfies
$$
\sup \big\{ \|\sigma_{\D_2}(x,\xi)\| \;\big|\; (x,\xi)\in T^*M,\;g(\xi,\xi)\leq1\big\} < \infty .
$$
\begin{prop}[cf.\ {\cite[\S8]{KL13}}]
\label{prop:parameter_triple}
The operator $\D_2$ yields an odd spectral triple $(C_c^\infty(M), L^2(M,F), \D_2)$. 
\end{prop}
\begin{proof}
The completeness of $M$ and the bounded propagation 
speed ensure the essential self-adjointness of $\D_2$ 
on $\Gamma_c^\infty(M,F)$ (see e.g.\ \cite[Proposition 10.2.11]{Higson-Roe00}). 
Since $\D_2$ is a first-order differential operator, the commutator with a smooth, compactly supported function is bounded. 
Lastly, ellipticity of $\D_2$ 
ensures that its resolvent is locally compact (see e.g.\ \cite[Proposition 10.5.2]{Higson-Roe00}).
\end{proof}

We are now ready to construct the odd unbounded Kasparov product of the $C_0(M,A)$-$C_0(M)$-module $(\A\odot C_c^\infty(M), C_0(M,\mH),\D_1(\cdot))$ with the $C_0(M)$-$\C$-module $(C_c^\infty(M), L^2(M,F),\D_2)$. 
On the internal tensor product of the Hilbert modules $C_0(M,\mH) \otimes_{C_0(M)} L^2(M,F) \simeq L^2(M,\mH\otimes F)$ we consider the operator $\D_1(\cdot)\otimes 1$, which we simply denote as $\D_1(\cdot)$ on $L^2(M,\mH\otimes F)$. Using the identification $L^2(M,\mH\otimes F) \simeq \mH\otimes L^2(M,F)$, we also consider the operator $1\otimes\D_2$, which we simply denote as $\D_2$. 
\begin{thm}
\label{thm:DS_triple}
Let $M$ be a complete oriented Riemannian manifold of dimension $m$, and let $\mH$ be a separable Hilbert space. 
Let $\D_2$ be a closed first-order symmetric elliptic differential operator on a hermitian vector bundle $F\to M$, which has bounded propagation speed. 
Let $\big\{(\A,{}_{\pi_x}\mH,\D_1(x))\big\}_{x\in M}$ be a weakly differentiable family of spectral triples. 
Then the following statements hold:
\begin{enumerate}
\item the operator
$$
\D_1\times\D_2 := \mattwo{0}{\D_1(\cdot) - i\D_2}{\D_1(\cdot) + i\D_2}{0} \colon \left(\Dom(\D_1(\cdot))\cap\Dom(\D_2)\right)^{\oplus2} \to L^2(M,\mH\otimes F)^{\oplus2} 
$$
yields an even spectral triple $(\A\odot C_c^\infty(M) , L^2(M,\mH\otimes F)^{\oplus2} , \D_1\times\D_2)$ which represents the odd unbounded Kasparov product of $(\A\odot C_c^\infty(M), C_0(M,\mH)_{C_0(M)},\D_1(\cdot))$ with $(C_c^\infty(M), L^2(M,F),\D_2)$;
\item the triple $(\A\odot C_c^\infty(M), L^2(M,\mH\otimes F) , \D_1(\cdot) + i \D_2)$ is an odd indefinite spectral triple. 
\end{enumerate}
\end{thm}
\begin{proof}
For the first statement, we need to show that we have a correspondence (as defined in \cite[Definition 6.3]{KL13}) from $(\A\odot C_c^\infty(M), C_0(M,\mH)_{C_0(M)},\D_1(\cdot))$ to $(C_c^\infty(M), L^2(M,F),\D_2)$, so that we can apply \cite[Theorems 6.7 \& 7.5]{KL13}. For a family of operators, this has been shown in \cite[Proposition 8.11]{KL13}. For a family of spectral triples, the only difference is that we now consider a left action by $C_0(M,A)$ (instead of $\C$) on $C_0(M,\mH)$. Thus we need to check the third condition in 
\cite[Definition 6.3]{KL13},
which requires that the commutator $[\D_2,\pi(a\otimes f)\otimes1]\colon\Dom(\D_2)\to C_0(M,\mH) \otimes_{C_0(M)} L^2(M,F) \simeq L^2(M,\mH\otimes F)$ is well-defined and bounded for all $a\otimes f\in \A\odot C_c^\infty(M)$. 

The commutator with $f\in C_c^\infty(M)$ simply yields $[\D_2,f] = \sigma_{\D_2}(df)$, which is bounded because $f\in C_c^\infty(M)$ implies that $df$ is bounded, and because $\sigma_{\D_2}$ is completely bounded by \cite[Proposition 8.2]{KL13}. Similarly, the commutator $[\D_2,\pi_x(a)] = \sigma_{\D_2}(d(\pi_x(a)))$ is bounded, because by assumption the weak derivative of $\pi_x(a)$ is uniformly bounded. Thus we indeed have a correspondence, and the first statement then follows from \cite[Theorems 6.7 \& 7.5]{KL13}. 

For the second statement, consider the operator $\D := \D_1(\cdot) + i\D_2$
on $\Dom\D = \Dom\D_1(\cdot)\cap\Dom\D_2$. We know from 
\cite[Proposition 8.11]{KL13}
that $(\D_1(\cdot),\D_2)$ is an almost commuting pair, so it follows from 
\cref{thm:sum-sa-reg} 
that $\D^* = \D_1(\cdot)-i\D_2$ on $\Dom\D^* = \Dom\D_1(\cdot)\cap\Dom\D_2$, and therefore we have $\Re\D = \D_1(\cdot)$ and $\Im\D = \D_2$ on this domain. 

The operators $\D_1(\cdot)$ and $\D_2$ are both essentially self-adjoint on $\Dom\D_1(\cdot)\cap\Dom\D_2$ (for $\D_1(\cdot)$ this follows from \cref{lem:pair_S_ess-sa}, and for $\D_2$ this follows from the completeness of the Riemannian manifold). The domain $\Dom\D_1(\cdot)\cap\Dom\D_2$ is preserved by $\A\odot C_c^\infty(M)$, and both $\D_1(\cdot)$ and $\D_2$ have bounded commutators with $\A\odot C_c^\infty(M)$. Lastly, the inclusion $\iota\colon \Dom\D_1(\cdot)\cap\Dom\D_2\into L^2(M,\mH\otimes F)$ is locally compact because (by the first statement) $\D_1\times\D_2$ has locally compact resolvent. 
\end{proof}
\begin{remark}
In the construction of the 
operator $\D_1\times\D_2$ we may replace $\D_2$ by $-\D_2$, without affecting the first statement of the above theorem. We thus obtain two different spectral triples with the operators 
\begin{align*}
\mattwo{0}{\D_1(\cdot) - i\D_2}{\D_1(\cdot) + i\D_2}{0}
\quad\text{and}\quad 
\mattwo{0}{\D_1(\cdot) + i\D_2}{\D_1(\cdot) - i\D_2}{0} .
\end{align*}
The second statement of the theorem could have been proved alternatively by showing that these two spectral triples form a \emph{pair} of spectral triples. It then follows from \cref{thm:odd_indef-even_pair} that $\D = \D_1(\cdot) + i\D_2$ yields an odd indefinite spectral triple. 
\end{remark}
\subsubsection{Generalised Lorentzian cylinders}
Dirac operators on generalised pseudo-Riemannian cylinders have been studied in \cite{BGM05}. Here we will specialise to the Lorentzian case, and we will show that this provides an example of a family of spectral triples parametrised by the real line $\R$. 

Let $\Sigma$ be an $(n-1)$-dimensional smooth spin manifold, and let $g_t$ be a smooth family of \emph{complete} Riemannian metrics on $\Sigma$ parametrised by $t\in\R$. Consider the \emph{generalised Lorentzian cylinder} $(M,g) := (\Sigma\times\R, g_t - dt^2)$. The vector field $\nu := \partial_t$ is a unit timelike vector field which is orthogonal to the hypersurfaces $\Sigma_t := (\Sigma\times\{t\},g_t)$. 

Since each hypersurface $\Sigma_t$ is a complete Riemannian spin manifold, we obtain for each $t\in\R$ a spectral triple
$$
\big( C_c^\infty(\Sigma), L^2(\Sigma_t,\bS_t), \sD(t) \big) ,
$$
where $\bS_t$ is the spinor bundle over $\Sigma_t$, and $\sD(t) = \gamma_t\circ\nabla^{\bS_t}$ is the canonical Dirac operator on $\Sigma_t$. 

For $x\in\Sigma$ and $t_0,t_1\in\R$, parallel transport along the curve $t\mapsto (t,x)\in M$ (i.e.\ an integral curve of the vector field $\nu$) yields a linear isometry $\tau_{t_0}^{t_1}\colon(\bS_{t_0})_x\to(\bS_{t_1})_x$. 
The Hilbert spaces $\mH_t := L^2(\Sigma_t,\bS_t)$ of square-integrable spinors on $\Sigma_t$ can be identified via this parallel transport, and we shall write $\mH := \mH_0$. Under this identification, the action of $C_c^\infty(\Sigma)$ on $\mH_t\simeq\mH$ (given by pointwise multiplication) does not depend on $t$. 

A local orthonormal frame $\{e_1,\ldots,e_{n-1}\}$ on $\Sigma_0$ can be extended to an orthonormal frame $\{\nu,e_1,\ldots,e_{n-1}\}$ on $M$ via parallel transport along $\nu$, and this extended frame then satisfies $\nabla_\nu e_j = 0$. Consequently, the Clifford multiplication $\gamma$ on $(M,g)$ satisfies 
$$
\big[\nabla^S_\nu , \gamma(e_j)\big] = \gamma(\nabla_\nu e_j) = 0 ,
$$
so $\gamma$ is parallel along the vector field $\nu$. Under the identification 
$\tau_t^0\colon \mH_t \to \mH_0$, 
the Clifford multiplication $\gamma_t$ on $\mH_t$ is mapped to $\tau_t^0 \circ \gamma_t(\tau_0^t X) \circ \tau_0^t = \gamma_0(X)$ on $\mH_0$ (see also \cite[\S5]{BGM05}). Thus, upon identifying $\mH_t \simeq \mH_0$, the Clifford multiplication becomes independent of $t$. 

\begin{prop}
Let $(M,g)$ be an even-dimensional generalised Lorentzian cylinder as constructed above. Suppose that the smooth family of metrics $g_t$ has derivatives of all orders (both in $t$ and along $\Sigma$) which are \emph{globally bounded}. Then the spectral triples $\big( C_c^\infty(\Sigma), L^2(\Sigma_t,\bS_t), \sD(t) \big)$ form a weakly differentiable family of spectral triples (as in \cref{defn:family}) parametrised by the real line $M=\R$. 
\end{prop}
\begin{proof}
We define the Hilbert space $W := \Dom\sD(0)$ equipped with the graph inner product of $\sD(0)$. Then $W$ is continuously and densely embedded in $\mH := L^2(\Sigma_0,\bS_0)$. Since $\sD(0)$ is elliptic, this embedding is locally compact. 

Using the fact that $\gamma_t$ is independent of $t$ under the identification $L^2(\Sigma_t,\bS_t) \simeq \mH$, we can write $\sD(t) - \sD(0) = \gamma_0\circ(\nabla^{\bS_t}-\nabla^{\bS_0})$, which is a smooth endomorphism on $\bS_0$. The assumption that $g_t$ has globally bounded derivatives ensures that $\sD(t) - \sD(0)$ is globally bounded, and therefore the graph norms of $\sD(t)$ are uniformly equivalent.

For $f\in C_c^\infty(\Sigma)$, the commutator $[\sD(t),f]$ is given by Clifford multiplication with $df$. Hence, under the identification $L^2(\Sigma_t,\bS_t) \simeq \mH$, both $f$ and $[\sD(t),f]$ are independent of $t$. Lastly, since $g_t$ has globally bounded derivatives, it follows from \cite[Theorem 5.1]{BGM05} that the time-derivative of $\sD(t)$ is relatively bounded by $\sD(t)$ (and hence by $\sD(0)$). 
\end{proof}

By considering $\D_2 = -i\partial_x$ on $L^2(\R)$, \cref{thm:DS_triple} then yields the odd \emph{indefinite} spectral triple 
$$
\big( C_c^\infty(\Sigma\times\R), L^2(\R,\mH), \sD(\cdot) + \partial_t \big) ,
$$
describing the Dirac operator on the foliated spacetime $\Sigma\times\R$. In fact, this example provided our initial motivation to consider families of spectral triples, and we intend to study this approach to foliated spacetimes in more detail in a future work.

\newcommand{\MR}[1]{}


\providecommand{\noopsort}[1]{}\providecommand{\vannoopsort}[1]{}
\providecommand{\bysame}{\leavevmode\hbox to3em{\hrulefill}\thinspace}
\providecommand{\MR}{\relax\ifhmode\unskip\space\fi MR }
\providecommand{\MRhref}[2]{%
  \href{http://www.ams.org/mathscinet-getitem?mr=#1}{#2}
}
\providecommand{\href}[2]{#2}

\end{document}